\documentclass[preprint]{imsart}
\usepackage{graphicx,epsfig}
\usepackage{multirow}
\usepackage{multicol}
\usepackage{enumerate}
\RequirePackage{natbib}

\usepackage{mathrsfs}
\usepackage[normalem]{ulem}
\usepackage{algorithm}
\usepackage{algorithmic}
\usepackage{subfig}

\usepackage[pdftex,bookmarks,colorlinks,breaklinks]{hyperref}
\hypersetup{linkcolor=blue,citecolor=blue,filecolor=cyan,urlcolor=blue}
\usepackage{breakcites}
\usepackage{verbatim}

\usepackage[usenames]{color}

\usepackage{parskip}

\usepackage{amsmath,amssymb,latexsym, amsfonts, amscd, amsthm, xy}
\input{xy}
\xyoption{all}

\input{xy}
\xyoption{all}

\arxiv{arXiv:1403.2660}

\startlocaldefs

\theoremstyle{plain}
\newtheorem{theorem}{Theorem}[section]

\newtheorem{proposition}[theorem]{Proposition}
\newtheorem{corollary}[theorem]{Corollary}

\newtheorem{assumption}[theorem]{Assumption}

\newtheorem{lemma}{Lemma}[section]

\theoremstyle{definition}

\newtheorem{remark}[theorem]{Remark}
\newtheorem{example}[theorem]{Example}
\numberwithin{equation}{section}

\def\r{\right}
\def\l{\left}
\newcommand{\eps}{\varepsilon}
\newcommand{\pb}[1]{P{\left(#1\right)}}

\newcommand{\dotp}[2]{\left\langle#1,#2\right\rangle}
\newcommand{\m}{\mathcal}
\newcommand{\mb}{\mathbb}
\newcommand\argmin{\mathop{\mbox{argmin}}}

\newcommand{\med}{\mbox{{\rm med}}}
\newcommand{\st}{\mbox{\rm st}}

\newcommand{\sto}{\mathrm{st}}

\input{./Definitions}
\endlocaldefs

\begin{document}

\begin{frontmatter}
\title{Robust and scalable Bayes via a median of subset posterior measures}
\runtitle{Robust and Scalable Bayes}

\begin{aug}

\author{\fnms{Stanislav} \snm{Minsker}\thanksref{t2,t3}\ead[label=e1]{minsker@usc.edu}},
\author{\fnms{Sanvesh} \snm{Srivastava}\thanksref{t2}\ead[label=e2]{sanvesh-srivastava@uiowa.edu}},
\author{\fnms{Lizhen} \snm{Lin}\thanksref{t2}\ead[label=e3]{lizhen.lin@austin.utexas.edu}}
\and
\author{\fnms{David} \snm{Dunson}\thanksref{t2}
\ead[label=e4]{dunson@duke.edu}}

\thankstext{t2}{Authors were partially supported by grant R01-ES-017436 from the National Institute of Environmental Health Sciences (NIEHS) of the National Institutes of Health (NIH).}
\thankstext{t3}{Stanislav Minsker acknowledges support from NSF grants FODAVA CCF-0808847, DMS-0847388, ATD-1222567.}
\runauthor{S. Minsker et al.}


\address{Stanislav Minsker\\
Department of Mathematics\\ 
University of Southern California\\
Los Angeles, CA 90089\\
\printead{e1}}

\address{Sanvesh Srivastava\\
Department of Statistics and Actuarial Science\\
The University of Iowa\\
Iowa City, IA 52242\\
 \printead{e2}}

\address{Lizhen Lin\\
Department of Statistics and Data Sciences\\
The University of Texas at Austin\\
Austin, TX 78712\\
\printead{e3}}

\address{David B. Dunson\\
Department of Statistical Science\\
Duke University, Box 90251\\
Durham NC 27708\\
\printead{e4}}

\end{aug}

\begin{abstract}
We propose a novel approach to Bayesian analysis that is provably robust to outliers in the data and often has computational advantages over standard methods. Our technique is based on splitting the data into non-overlapping subgroups, evaluating the posterior distribution given each independent subgroup, and then combining the resulting measures. The main novelty of our approach is the proposed aggregation step, which is based on the evaluation of a median in the space of probability measures equipped with a suitable collection of distances that can be quickly and efficiently evaluated in practice.  
We present both theoretical and numerical evidence illustrating the improvements achieved by our method.
\end{abstract}

\begin{keyword}[class=MSC]
\kwd[Primary ]{62F15}
\kwd[; secondary ]{68W15}
\kwd{62G35}
\end{keyword}

\begin{keyword}
\kwd{Big data}
\kwd{Distributed computing}
\kwd{Geometric median}
\kwd{Parallel MCMC}
\kwd{Stochastic approximation}
\kwd{Wasserstein distance.}
\end{keyword}

\end{frontmatter}

\section{Introduction}

Contemporary data analysis problems pose several general challenges.
One is resource limitations: massive data require computer clusters for storage and processing.
Another problem occurs when data are severely contaminated by ``outliers''  that are not easily identified and removed. Following \cite{Box1968A-Bayesian-appr00}, an outlier can be defined as \emph{``being an observation which is suspected to be partially or wholly irrelevant because it is not generated by the stochastic model assumed.''}  While the topic of robust estimation has occupied an important place in the statistical literature for several decades and
significant progress has been made in the theory of point estimation, robust Bayesian methods are not sufficiently well-understood.

Our main goal is to make a step towards solving these problems, proposing a general Bayesian approach that is 
\begin{enumerate}[(i)] 
\item provably robust to the presence of outliers in the data without any specific assumptions on their distribution or reliance on preprocessing; 
\item scalable to big data sets through allowing computational algorithms to be implemented in parallel for different data subsets prior to an efficient aggregation step.  
\end{enumerate}
The proposed approach consists in splitting the sample into disjoint parts, implementing Markov chain Monte Carlo (MCMC) or another posterior sampling method to obtain draws from each ``subset posterior'' in parallel, and then using these draws to obtain weighted samples from the {\em median posterior (or M-Posterior)}, a new probability measure which is a (properly defined) median of a collection of subset posterior distributions. 
We show that, despite the loss of ``interactions'' among the data in different groups, the final result still admits strong guarantees; moreover, splitting the data gives certain advantages in terms of robustness to outliers. 

In particular, we demonstrate that the M-posterior is a probability measure centered at the ``robust'' estimator of the unknown parameter, 
the associated credible sets are often of the same ``width'' as the credible sets obtained from the usual posterior distribution
and admit strong ``frequentist'' coverage guarantees (see section \ref{sec:approx} for exact statements). 

The paper is organized as follows: section \ref{sec:literature} contains an overview of the existing literature and explains the goals that we aim to achieve in this work.
Section \ref{sec:prelim} introduces the mathematical background and key facts used throughout the paper.
Section \ref{sec:main} describes the main theoretical results for the median posterior.
Section \ref{sec:ex} presents details of algorithms, implementation, and numerical performance of the median posterior for several models. 
The simulation study and analysis of data examples convincingly show the robustness properties of the median posterior.  
We have also implemented the matrix completion example based on the MovieLens data set \citep{Gro13} which illustrates the scalability of the model. 
Proofs that are omitted in the main text are contained in the appendix.

\subsection{Discussion of related work}
\label{sec:literature}

A. Dasgupta remarks that (see the discussion following \cite{Berger1994An-overview-of-00}): ``Exactly what constitutes a study of Bayesian robustness is of course impossible to define.''
 The popular definition (which also indicates the main directions of research in this area) is due to J. Berger \citep{Berger1994An-overview-of-00}: ``Robust Bayesian analysis is the study of the sensitivity of Bayesian answers to uncertain inputs. These uncertain inputs are typically the model, prior distribution, or utility function, or some combination thereof.''
 Outliers are typically accommodated by either employing heavy-tailed likelihoods (e.g., \cite{Svensen2005Robust-Bayesian00}) or by attempting to identify and remove them as a first step (as in \cite{Box1968A-Bayesian-appr00} or \cite{Bayarri1994Robust-Bayesian00}).
 The usual assumption in the Bayesian literature is that the distribution of the outliers can be modeled (e.g., using a $t$-distribution, contamination by a larger variance parametric distribution, etc).  In this paper, we instead bypass the need to place a model on the outliers and do not require their removal prior to analysis.  We base inference on the \emph{median posterior}, whose robustness can be  formally and precisely quantified in terms of  concentration properties  around the true  delta measure  under the potential  influence of outliers and contaminations  of arbitrary nature. 


Also relevant is the recent progress in scalable Bayesian algorithms.
Most methods designed for distributed computing share a common feature: they efficiently use the data subset available to a single machine and combine the ``local'' results for ``global'' learning, while minimizing communication among cluster machines (\cite{SmoNar10}).
A wide variety of optimization-based approaches are available for distributed learning \citep{Boyetal11}; however, the number of similar Bayesian methods is limited. One of the reasons for this limitation is related to Markov chain Monte Carlo (MCMC), the dominating approach for approximating the posterior distribution of parameters in Bayesian models.
While there are many efficient MCMC techniques for sampling from posterior distributions based on small subsets of the data (called ``subset posteriors'' in the sequel), to the best of our knowledge, there is no general rigorously justified approach for combining the subset posteriors into a single distribution for improved performance.

Three major approaches exist for scalable Bayesian learning in a distributed setting. The first approach independently evaluates the likelihood for each data subset across multiple machines and returns the likelihoods to a ``master'' machine, where they are appropriately combined with the prior using conditional independence assumptions of the probabilistic model.
These two steps are repeated at every MCMC iteration (see \cite{SmoNar10,AgaDuc12}).
This approach is problem-specific and involves extensive communication among machines. The second approach uses a so-called stochastic approximation (SA) and successively learns ``noisy'' approximations to the full posterior distribution using data in small mini-batches.
The accuracy of SA increases as it uses more data.
A group of methods based on this approach uses sampling-based techniques to explore the posterior distribution through modified Hamiltonian or Langevin dynamics (e.g., \cite{WelTeh11,AhnKorWel12,KorCheWel13}).
Unfortunately, these methods fail to accommodate discrete-valued parameters and multimodality.
Another subgroup of methods uses deterministic variational approximations and learns the variational parameters of the approximated posterior through an optimization-based approach (see \cite{Wanetal11, Hofetal13, Broetal13}).
Although these techniques often have excellent predictive performance, it is well known \citep{Bishop2006Pattern-recogni00} that variational methods tend to substantially underestimate posterior uncertainty and provide a poor characterization of posterior dependence, while lacking theoretical guarantees.

Our approach instead falls in a third class of methods which avoid extensive communication among machines by running independent MCMC chains for each data subset and obtaining draws from subset posteriors.
These subset posteriors can be combined in a variety of ways.
Some of these methods simply average draws from each subset \citep{Scoetal13}.
Other alternatives use an approximation to the full posterior distribution based on kernel density estimates \citep{NeiWanXin13} or the so-called Weierstrass transform \citep{WanDun13}.
These methods have limitations related to the dimension of the parameter, moreover, their applicability and theoretical justification are restricted to parametric models. Unlike the method proposed below, none of the aforementioned algorithms are provably robust.

Our work was inspired by recent multivariate median-based techniques for robust estimation developed in \cite{Minsker2013Geometric-Media00} (see also \cite{Hsu2013Loss-minimizati00,alon1996space,lerasle2011robust,Nemirovski1983Problem-complex00} where similar ideas were applied in different frameworks).

\section{Preliminaries}
\label{sec:prelim}

We proceed by recalling key definitions and facts which will be used throughout the paper.

\subsection{Notation}

In what follows, $\|\cdot\|_2$  denotes the standard Euclidean distance in $\mb R^p$ and $\dotp{\cdot}{\cdot}_{\mb R^p}$ - the associated dot product.

Given a totally bounded metric space $(\mb Y,d)$, the packing number $M(\eps,\mb Y,d)$ is the maximal number $N$ such that there exist $N$ \emph{disjoint}
$d$-balls $B_1,\ldots,B_N$ of radius $\eps$ contained in $\mb Y$, i.e., $\bigcup\limits_{j=1}^N B_j\subseteq \mb Y$.

Let $\{p_\theta, \ \theta\in\Theta\}$ be a family of probability density functions on $\mb R^p$.
Let $l,u:\mb R^p\mapsto \mb R_+$ be two functions such that $l(x)\leq u(x)$ for every $x\in \mb R^p$ and
$d^2(l,u):=\int\limits_{\mb R^p}(\sqrt{u}-\sqrt{l})^2(x)dx<\infty$.
A bracket $[l,u]$ consists of all functions $g:\mb R^p\mapsto \mb R$ such that $l(x)\leq g(x)\leq u(x)$ for all $x\in \mb R^p$.
For $A\subseteq \Theta$, the bracketing number $N_{[\,]}(\eps,A,d)$ is defined as the smallest number $N$ such that there exist $N$ brackets $[l_i,u_i], \ i=1,\ldots, N$ satisfying
$\{p_\theta, \ \theta\in A\}\subseteq \bigcup\limits_{i=1}^N [l_i,u_i]$ and $d(l_i,u_i)\leq \eps$ for all $1\leq i\leq N$.

For $y\in \mb Y$, $\delta_y$ denotes the Dirac measure concentrated at $y$.
In other words, for any Borel-measurable $B$, $\delta_y(B)=I\{y\in B\}$, where $I\{\cdot\}$ is the indicator function.

We will say that $k:\mb Y\times \mb Y\mapsto \mb R$ is a \emph{kernel} if it is a symmetric, positive definite function.
Assume that $\left(\mb H, \dotp{\cdot}{\cdot}_\mb H\right)$ is a reproducing kernel Hilbert space (RKHS) of functions $f:\mb Y\mapsto \mb R$.
Then $k$ is a \emph{reproducing kernel} for $\mb H$ if for any $f\in \mb H$ and $y\in \mb Y$, $\dotp{f}{k(\cdot,y)}_\mb H=f(y)$ (see \cite{Aronszajn1950Theory-of-repro00} for details).

For a square-integrable function $f\in L_2(\mb R^p)$, $\hat f$ stands for its Fourier transform.
For $x\in \mb R$, $\lfloor x\rfloor$ denotes the largest integer not greater than $x$.

Finally, given two nonnegative sequences $\{a_n\}$ and $\{b_n\}$, we write $a_n\lesssim b_n$ if $a_n\leq C b_n$ for some $C>0$ and all $n$.
Other objects and definitions are introduced in the course of exposition when necessity arises.
	
\subsection{Generalizations of the univariate median}

\label{sec:median}

Let $\mb Y$ be a normed space with norm $\|\cdot\|$, and let  $\mu$ be a probability measure on $(\mb Y,\|\cdot\|)$ equipped with Borel $\sigma$-algebra.
Define the \textit{geometric median} of $\mu$ by
\[
x_\ast=\argmin\limits_{y\in \mb Y} \int_\mb Y \l(\|y-x\|-\|x\|\r)\mu(dx).
\]
In this paper, we focus on the special case when $\mu$ is a uniform distribution on a finite collection of atoms
$x_1,\ldots,x_m\in \mb Y$, so that
\begin{align}
\label{eq:median}
x_\ast=\med_g(x_1,\ldots,x_m):=\argmin\limits_{y\in \mb Y}\sum_{j=1}^m \|y-x_j\|.
\end{align}
The geometric median exists under rather general conditions; for example, if $\mb Y$ is a Hilbert space (this case will be our main focus, for more general conditions see \cite{kemperman1987median}).
Moreover, it is well-known that in this situation $x_\ast\in {\rm co}(x_1,\ldots,x_m)$ -- the convex hull of $x_1,\ldots,x_m$ (meaning that there exist nonnegative $\alpha_j, \  j=1\ldots m, \ \sum\limits_{j=1}^m \alpha_j=1$ such that $x_\ast=\sum\limits_{j=1}^m \alpha_j x_j$).

Another useful generalization of the univariate median is defined as follows.
Let $(\mb Y,d)$ be a metric space with metric $d$, and $x_1,\ldots,x_k\in \mb Y$.
Define $B_\ast$ to be the $d$-ball of minimal radius such that it is centered at one of $\{x_1,\ldots,x_m\}$ and contains at least  half of these points.
Then the median ${\rm med}_0(x_1,\ldots,x_m)$ of $x_1,\ldots,x_m$ is the center of $B_\ast$.
In other words, let
\begin{align}
\label{eq:eps_ast}
\eps_\ast:=\inf\Big\{&\eps>0: \ \exists  j=j(\eps)\in\{1,\ldots, m\} \text{ and } I(j)\subset \{1,\ldots ,m\}  \text{ such that }\\
&\nonumber
|I(j) |>\frac m 2 \text{ and }
\forall i\in I(j), \ d(x_i,x_j)\leq 2\eps\Big\},
\end{align}
$
j_\ast:=j(\eps_\ast), \text{ where ties are broken arbitrarily, }
$
and set
\begin{align}
&
\label{eq:nemir0}
x_\ast={\rm med}_0 (x_1,\ldots,x_m):=x_{j_\ast}.
\end{align}
We will say that $x_\ast$ is the \textit{metric median} of $x_1,\ldots,x_m$.
Note that $x_\ast$ always belongs to $\{x_1,\ldots,x_m\}$ by definition.
Advantages of this definition are its generality (only metric space structure is assumed) and simplicity of numerical evaluation since only the pairwise distances $d(x_i,x_j), \ i,j=1,\ldots, m$ are required  to compute the median.
This construction was previously employed in \cite{Nemirovski1983Problem-complex00} in the context of stochastic optimization and is further studied in \cite{Hsu2013Loss-minimizati00}. 
A closely related notion of the median was used in \cite{Lopuhaa1991Breakdown-point00} under the name of the ``minimal volume ellipsoid'' estimator.

Finally, we recall an important property of the median (shared both by $\med_g$ and $\med_0$) which states that it transforms a collection of independent, ``weakly concentrated'' estimators into a single estimator with significantly stronger concentration properties.
Given $q,\alpha$ such that $0<q<\alpha<1/2$, define a nonnegative function $\psi(\alpha,q)$ via 
\begin{align}
\label{eq:psi}
\psi(\alpha,q):=(1-\alpha)\log \frac{1-\alpha}{1-q}+\alpha\log\frac{\alpha}{q}.
\end{align}

The following result is an adaptation of Theorem 3.1 in \cite{Minsker2013Geometric-Media00}:

\begin{theorem}
\label{thm:main}
\text{}

\begin{description}
\item[a] Assume that $(\mb H,\|\cdot\|)$ is a Hilbert space and $\theta_0\in \mb H$.
Let $\hat\theta_1,\ldots,\hat\theta_m\in \mb H$ be a collection of independent random variables.
Let $\kappa$ be a constant satisfying $0\leq\kappa < \frac{1}{3}$. 
Suppose $\eps>0$ is such that for all $j, \ 1\leq j\leq \lfloor(1-\kappa)m\rfloor+1$,
\begin{align}
\label{eq:weak_con1}
\Pr\Big(\|\hat\theta_j-\theta_0\|>\eps\Big)\leq \frac{1}{7}.
\end{align}
Let $\hat\theta_\ast=\med_g(\hat\theta_1,\ldots,\hat\theta_m)$ be the \textit{geometric median} of
$\{\hat\theta_1,\ldots,\hat\theta_m\}$.
Then
\begin{align*}
\Pr\Big(\|\hat\theta_\ast-\theta_0\|> 1.52\eps\Big)\leq \left[e^{(1-\kappa)\psi\left(\frac{3/7-\kappa}{1-\kappa},1/7\right)}\right]^{-m}.
\end{align*}
\item[b]
Assume that $(\mb Y,d)$ is a metric space and $\theta_0\in \mb Y$.
Let $\hat\theta_1,\ldots,\hat\theta_m\in \mb Y$ be a collection of independent random variables.
Let $\kappa$ be a constant satisfying $0\leq\kappa < \frac{1}{3}$. 
Suppose $\eps>0$ is such that for all $j,  \ 1\leq j\leq \lfloor(1-\kappa)m\rfloor+1$,
\begin{align}
\label{eq:weak_con2}
\Pr\Big(d(\hat\theta_j,\theta_0)>\eps\Big)\leq \frac{1}{4}.
\end{align}
Let $\hat\theta_\ast=\med_0(\hat\theta_1,\ldots,\hat\theta_m)$. 
Then
\begin{align*}
\Pr\Big(d(\hat\theta_\ast,\theta_0)> 3\eps\Big)\leq \left[e^{(1-\kappa)\psi\left(\frac{1/2-\kappa}{1-\kappa},1/4\right)}\right]^{-m}.
\end{align*}
\end{description}

\end{theorem}
\begin{proof}
See section \ref{proof:med}. 
\end{proof}
\begin{remark}
While we require $\kappa<1/3$ above for clarity and to keep the constants small, we prove a slightly more general result that holds for any $\kappa < 1/2$.
\end{remark}

Theorem \ref{thm:main} implies that the concentration of the geometric median of independent estimators around the ``true'' parameter value improves geometrically fast with respect to the number of such estimators, while the estimation rate is preserved, up to a constant. 
In our case, the role of $\hat\theta_j$'s will be played by posterior distributions based on disjoint subsets of observations, viewed as elements of the space of signed measures equipped with a suitable distance. 

Parameter $\kappa$ allows taking corrupted observations into account: if the initial sample contains not more than $\lfloor\kappa m\rfloor$ outliers (of arbitrary nature), then at most $\lfloor\kappa m\rfloor$ estimators amongst $\{\theta_1,\ldots,\theta_m\}$ can be affected but their median remains stable, still being close to the unknown $\theta_0$ with high probability. 
To clarify the notion of ``robustness'' that such a statement provides, assume that $\hat\theta_1,\ldots,\hat\theta_m$ are consistent estimators of $\theta_0$ based on disjoint samples of size $n/m$ each. 
If $\frac{n}{m}\to \infty$, then $\frac{\kappa m}{n}\to 0$, hence the breakdown point of the estimator $\hat\theta_\ast$ is $0$ is general. 
However, it is able to handle a number of outliers that grows like $o(n)$ while preserving consistency, which is the best one can hope for without imposing any additional assumptions on the underlying distribution, parameter of interest or nature of the outliers.

Let us also mention that the the geometric median of a collection of points in a Hilbert space belongs to the convex hull of these points. 
Thus, one can think about ``downweighing'' some observations (potential outliers) and increasing the weight of others, and geometric median gives a way to formalize this approach. 
The median $\med_0$ defined in (\ref{eq:nemir0}) corresponds to the extreme case when all but one weight are equal to $0$. 
Its potential advantage lies in the fact that its evaluation requires only the knowledge of pairwise distances $d(\hat\theta_i,\hat\theta_j), \ i,j=1,\ldots, m$, see (\ref{eq:eps_ast}). 

\subsection{Distances between probability measures}
\label{sec:distances}

Next, we discuss the special family of distances between probability measures that will be used throughout the paper.
These distances provide the necessary structure to define and  evaluate medians in the space of measures, as discussed above.
Since one of our goals was to develop computationally efficient techniques, we focus on distances that admit accurate numerical approximation.

Assume that $(\mb X,\rho)$ is a separable metric space, and let $\m F=\{f: \mb X\mapsto \mb R\}$ be a collection of real-valued functions.
Given two Borel probability measures $P,Q$ on $\mb X$, define
\begin{align}
\label{eq:distance}
\|P-Q\|_\m F:=\sup_{f\in \m F}\l|\int_\mb X f(x) d(P-Q)(x)\r|.
\end{align}
Important special cases include the situation when
\begin{align}
\label{eq:f_l}
\m F=\m F_L:=\{f:\Theta\mapsto \mb R \text{ s.t. } \|f\|_L\leq 1\},
\end{align}
where
$\|f\|_L:=\sup\limits_{x_1\ne x_2}\frac{|f(x_1)-f(x_2)|}{\rho(x_1,x_2)}$ is the Lipschitz constant of $f$.

It is well-known (\cite{Dudley2002Real-analysis-a00}, Theorem 11.8.2) that in this case $\|P-Q\|_{\m F_L}$ is equal to the Wasserstein distance (also known as the Kantorovich-Rubinstein distance)
\begin{equation}
\label{eq:Wasser}
d_{W_{1,\rho}}(P,Q)=\inf\Big\{ \mb E \rho(\boldsymbol X,\boldsymbol Y): \m L(\boldsymbol X)= P, \ \m L(\boldsymbol Y) = Q\Big\} ,
\end{equation}
where $\m L(\boldsymbol Z)$ denotes the law of a random variable $\boldsymbol Z$ and the infimum on the right is taken over the set of all joint distributions of $(\boldsymbol X,\boldsymbol Y)$ with marginals $P$ and $Q$.

Another fruitful structure emerges when $\m F$ is a unit ball in a Reproducing Kernel Hilbert Space
$\left(\mb H, \dotp{\cdot}{\cdot}_\mb H\right)$ with a reproducing kernel $k:\mb X\times \mb X\mapsto \mb R$.
That is,
\begin{align}
\label{eq:f_k}
\m F=\m F_k:=\{f:\mb X\mapsto \mb R, \ \|f\|_\mb H:=\sqrt{\dotp{f}{f}_\mb H}\leq 1\}.
\end{align}
Let $\m P_k:=\{P \text{ is a probability measure}, \ \int_\mb X \sqrt{k(x,x)}dP(x)<\infty\}$,
and assume that $P,Q\in \m P_k$.
Theorem 1 in \cite{Sriperumbudur2010Hilbert-space-e00} implies that the corresponding distance between measures $P$ and $Q$ takes the form
\begin{align}
\label{eq:kern_dist}
\|P-Q\|_{\m F_k}=\left\|\int_\mb X k(x,\cdot)d(P-Q)(x)\right\|_\mb H.
\end{align}
It follows that $P\mapsto \int_\mb X k(x,\cdot)dP(x)$ is an embedding of $\m P_k$ into the Hilbert space $\mb H$ which can be seen as an application of the ``kernel trick'' in our setting.
The Hilbert space structure allows one to use fast numerical methods to approximate the geometric median, see section \ref{sec:ex} below. 

\begin{remark}
Note that when $P$ and $Q$ are discrete measures (e.g., $P=\sum\limits_{j=1}^{N_1} \beta_j \delta_{z_j}$ and $Q=\sum\limits_{j=1}^{N_2} \gamma_j \delta_{y_j}$), then
\begin{align}
\label{eq:discrete}
\|P-Q\|^2_{\m F_k}&
=\sum_{i,j=1}^{N_1}\beta_i\beta_j k(z_i,z_j)+ \\
&\nonumber
\sum_{i,j=1}^{N_2}\gamma_i\gamma_j k(y_i,y_j)-2\sum_{i=1}^{N_1}\sum_{j=1}^{N_2}\beta_i \gamma_j k(z_i,y_j).
\end{align}
\end{remark}

In this paper, we will only consider \textit{characteristic} kernels, which means that $\|P-Q\|_{\m F_k}=0$ if and only if $P=Q$.
It follows from Theorem 7 in \cite{Sriperumbudur2010Hilbert-space-e00} that a sufficient condition for $k$ to be characteristic is its
\textit{strict positive definiteness}: we say that $k$ is \textit{strictly positive definite} if it is bounded, measurable, and such that for all non-zero signed Borel measures $\nu$
\[
\iint\limits_{\mb X\times \mb X} k(x,y)d\nu(x)d\nu(y)>0.
\]
When $\mb X=\mb R^p$, a simple sufficient criterion for the kernel $k$ to be characteristic follows from Theorem 9 in \cite{Sriperumbudur2010Hilbert-space-e00}:
\begin{proposition}
\label{prop:kernel}
Let $\mb X=\mb R^p, \ p\geq 1$.
Assume that $k(x,y)=\phi(x-y)$ for some bounded, continuous, integrable, positive-definite function $\phi:\mb R^p\mapsto \mb R$.
\begin{enumerate}
\item Let $\widehat \phi$ be the Fourier transform of $\phi$.
	If $|\widehat\phi(x)|>0$ for all $x\in\mb R^p$, then $k$ is characteristic;
\item If $\phi$ is compactly supported, then $k$ is characteristic.
\end{enumerate}
\end{proposition}
\begin{remark}
It is important to mention that in practical applications, we often deal with \textit{empirical measures} based on a collection of  MCMC  samples from the posterior distribution.
A natural question is the following: if $P$ and $Q$ are probability measures on $\mb R^D$ and $P_m$, $Q_n$ are their empirical versions,
what is the size of the error
\[
e_{m,n}:=\Big|\|P-Q\|_{\m F_k}-\|P_m-Q_n\|_{\m F_k}\Big|?
\]
For i.i.d samples, a useful and favorable fact is that $e_{m,n}$ often does not depend on $D$: under weak assumptions on kernel $k$, $e_{m,n}$ has an upper bound of order $m^{-1/2}+n^{-1/2}$ (that is, $\lim_{m,n\to \infty}\Pr\l(e_{m,n}\geq C(m^{-1/2}+n^{-1/2})\r)$ can be made arbitrarily small by choosing $C$ big enough, see Corollary 12 in \cite{Sriperumbudur2009On-integral-pro00}).
On the other hand, the bound for the (stronger) Wasserstein distance is not dimension-free and is of order $m^{-1/(D+1)}+n^{-1/(D+1)}$. Similar error rates hold for empirical measures based on samples from Markov Chains used to approximate invariant distributions, including MCMC samples (see \cite{boissard2014} and \cite{2013arXiv1312.2128F}).
\end{remark}
If $\mb X$ is a separable Hilbert space with dot product $\dotp{\cdot}{\cdot}_\mb X$ and $P_1,P_2$ are probability measures with 
\[
\int_\mb X \|x\|_\mb X dP_i(x)<\infty, \ i =1,2,
\]
it will be useful to assume that the class $\m F$ is chosen such that the distance between the measures is lower bounded by the distance between their means, namely
\begin{align}
\label{eq:mean}
\l\|\int_\mb X x dP_1(x)-\int_\mb X x dP_2(x)\r\|_\mb X\leq C\|P_1-P_2\|_\m F
\end{align}
for some absolute constant $C>0$. 
Clearly, this holds if $\m F$ contains the set of continuous linear functionals 
$\mb L=\l\{x\mapsto \dotp{u}{x}_\mb X, \ u\in \mb X, \ \|u\|_\mb X \leq 1/C\r\}$, since
\[
\l\|\int_\mb X x dP_i(x)\r\|_\mb X = \sup_{\|u\|_\mb X\leq 1}\int_\mb X \dotp{x}{u}_\mb X dP_i(x), \ i=1,2.
\]
In particular, this is true for the Wasserstein distance $d_{W_{1,\rho}}(\cdot,\cdot)$ defined with respect to the metric $\rho$ such that 
$\rho(x,y)\geq c_1 \|x-y\|_\mb X$. 
Next, we will state a simple sufficient condition on the kernel $k(\cdot,\cdot)$ for (\ref{eq:mean}) to hold for the unit ball $\m F_k$.
\begin{proposition}
\label{prop:mean}
Let $\mb X$ be a separable Hilbert space, $k_0:\mb X\times\mb X\mapsto \mb R$ - a characteristic kernel, and define 
\[
k(x,y):=k_0(x,y)+\dotp{x}{y}_\mb X.
\]
Then $k$ is characteristic and satisfies (\ref{eq:mean}) with $C=1$.
\end{proposition}
\begin{proof}
Let $\mb H_1$ and $\mb H_2$ be two reproducing kernel Hilbert spaces with kernels $k_1$ and $k_2$ respectively. 
It is well-known (e.g., \cite{Aronszajn1950Theory-of-repro00}) that the space corresponding to kernel $k=k_1+k_2$ is 
\[
\mb H=\{f=f_1+f_2, \ f_1\in \mb H_1,\ f_2\in \mb H_2\}
\]
with the norm $\|f\|^2_\mb H=\inf\{\|f_1\|^2_\mb H+\|f_2\|^2_\mb H, \ f_1+f_2=f\}$. 
Hence, the unit ball of $\mb H$ contains the unit balls of $\mb H_1$ and $\mb H_2$, so that for any probability measures $P,Q$
\[
\|P-Q\|_{\m F_k}\geq \max\l(\|P-Q\|_{\m F_{k_1}},\|P-Q\|_{\m F_{k_2}}\r),
\]
which easily implies the result.
\end{proof}
The kernels of the form $k(x,y)=k_0(x,y)+\dotp{x}{y}_\mb X$ will prove especially useful in the situation when the parameter of interest is finite-dimensional (see section \ref{sec:approx} for details). 

Finally, we recall the definition of the well-known Hellinger and total variation distances.
Assume that $P$ and $Q$ are probability measures on $\mb R^D$ which are absolutely continuous with respect to Lebesgue measure with densities $p$ and $q$ respectively.
Then the Hellinger distance between $P$ and $Q$ is given by
\[
h(P,Q):=\sqrt{\frac 1 2\int_{\mb R^D}\Big(\sqrt{p(x)}-\sqrt{q(x)}\Big)^2 dx}.
\]
The total variation distance between two probability measures defined on a $\sigma$-algebra $\mathfrak B$ is  
\[
\| P - Q \|_{{\mathrm TV}}=\sup_{B\in \mathfrak B}| P(B)-Q(B) |.
\]
		

\section{Contributions and main results}
\label{sec:main}

This section explains the construction of ``median posterior'' (or M-Posterior) distribution, along with the theoretical guarantees for its performance.

\subsection{Construction of robust posterior distribution}
\label{sec:construct}

Let $\{P_\theta, \ \theta\in \Theta\}$ be a family of probability distributions over $\mb R^D$ indexed by $\Theta$.
Suppose that for all $\theta\in \Theta$, $P_\theta$ is absolutely continuous with respect to Lebesgue measure $dx$ on $\mb R^D$ with $dP_\theta (\cdot)=p_\theta(\cdot)dx$.
In what follows, we equip $\Theta$ with a ``Hellinger metric''
\begin{align}
\label{eq:metric}
\rho(\theta_1,\theta_2):=h(P_{\theta_1},P_{\theta_2}),
\end{align}
and assume that the metric space $(\Theta,\rho)$ is separable. 

Let $k$ be a characteristic kernel defined on $\Theta\times \Theta$.
Kernel $k$ defines a metric on $\Theta$ via
\begin{align}
\label{eq:k_metric}
\rho_{k}(\theta_1,\theta_2):=\left\|k(\cdot,\theta_1)-k(\cdot,\theta_2)\right\|_\mb H=
\Big(k(\theta_1,\theta_1)+k(\theta_2,\theta_2)-2k(\theta_1,\theta_2)\Big)^{1/2},
\end{align}
where $\mb H$ is the RKHS associated to $k$.
We will assume that $(\Theta,\rho_k)$ is separable. 
Note that the ``Hellinger metric'' $\rho(\theta_1,\theta_2)$ is a particular case corresponding to the kernel 
\[
k_H(\theta_1,\theta_2):=\dotp{\sqrt{p_{\theta_1}}}{\sqrt{p_{\theta_2}}}_{L_2(dx)}. 
\]  
All subsequent results apply to this special case. 
While this is a ``natural'' metric for the problem, the disadvantage of $k_H(\cdot,\cdot)$ is that it is often difficult to evaluate numerically. 
Instead, we will consider metrics $\rho_k$ that are ``dominated'' by $\rho$ (this is formalized in assumption \ref{equiv}). 

Let $X_1,\ldots, X_n$ be i.i.d. $\mb R^D$-valued random vectors defined on a probability space $(\Omega,\m B, P)$ with unknown distribution $P_0:=P_{\theta_0}$ for some $\theta_0\in \Theta$.
Bayesian inference of $P_0$ requires specifying a prior distribution $\Pi$ over $\Theta$ (equipped with the Borel $\sigma$-algebra induced by $\rho$).
The posterior distribution given the observations $\m X_n:=\{X_1,\ldots, X_n\}$ is a random probability measure on $\Theta$ defined by
\[
\Pi_{n}(B|\m X_n):= \frac{\int\limits_B\prod_{i=1}^n p_\theta(X_i)d\Pi(\theta)}
{\int\limits_{\Theta}\prod_{i=1}^n p_\theta(X_i) d\Pi(\theta)}
\]
for all Borel measurable sets $B\subseteq \Theta$.
It is known (see \cite{Ghosal2000Convergence-rat00}) that under rather general assumptions the posterior distribution $\Pi_n$ ``contracts'' towards $\theta_0$, meaning that
\[
\Pi_n(\theta\in \Theta: \rho(\theta,\theta_0)\geq \eps_n|\m X_n)\to 0
\]
almost surely or in probability as $n\to \infty$ for a suitable sequence $\eps_n\to 0$.

One of the questions that we address can be formulated as follows: what happens if \textit{some} observations in $\m X_n$ are corrupted, e.g., if $\m X_n$ contains outliers of arbitrary nature and magnitude? 
Even if there is only one ``outlier'', the usual posterior distribution might concentrate most of its mass ``far'' from the true value $\theta_0$.

We proceed with a general description of our proposed algorithm for constructing a robust version of the posterior distribution.
Let $1\leq m\leq n/2$ be an integer.
Divide the sample $\m X_n$ into $m$ disjoint groups $G_1,\ldots, G_m$ of size
$|G_j|\geq\lfloor n/m\rfloor$ each:
\[
\{X_1,\ldots,X_n\}=\bigcup\limits_{j=1}^m G_j, \ G_i\cap G_l=\emptyset \text{ for } i\ne j, \ |G_j|\geq \lfloor n/m\rfloor, \ j=1\ldots m.
\]
A good choice of $m$  efficiently exploits the available computational resource while ensuring that the groups $G_j$s are sufficiently large.

Let $\Pi$ be a prior distribution over $\Theta$, and let 
\[
\Big\{\Pi^{(j)}(\cdot):=\Pi_{|G_j|}(\cdot|G_j), \ j=1,\ldots, m\Big\}
\]
be the family of subset posterior distributions depending on disjoint subgroups $G_j, \ j=1,\ldots, m$:
\[
\Pi_{|G_j|}(B|G_j):= \frac{\int\limits_B\prod_{i\in G_j} p_\theta(X_i)d\Pi(\theta)}
{\int\limits_{\Theta}\prod_{i\in G_j} p_\theta(X_i)d\Pi(\theta)}.
\]
Define the M-Posterior  as
\begin{align}
\label{eq:med_g}
&
\hat \Pi_{n,g}:=\med_g(\Pi^{(1)},\ldots,\Pi^{(m)}),
\end{align}
or
\begin{align}
\label{eq:med_0}
&
\hat \Pi_{n,0}:=\med_0(\Pi^{(1)},\ldots,\Pi^{(m)}),
\end{align}
where the medians $\med_g(\cdot)$ and $\med_0(\cdot)$ are evaluated with respect to $\|\cdot\|_{\m F_L}$ or $\|\cdot\|_{\m F_k}$ introduced in section \ref{sec:median} above.
Note that $\hat \Pi_{n,g}$ and $\hat\Pi_{n,0}$ are always probability measures: indeed, due to the aforementioned properties of a geometric median, there exists
$\alpha_1\geq 0,\ldots,\alpha_m\geq 0, \ \sum\limits_{j=1}^m \alpha_j=1$ such that
$\hat\Pi_{n,g}=\sum\limits_{j=1}^m \alpha_j \Pi^{(j)}$, and $\hat\Pi_{n,0}\in\{\Pi^{(1)}(\cdot),\ldots,\Pi^{(m)}(\cdot)\}$ by definition.

While $\hat \Pi_{n,g}$ and $\hat \Pi_{n,0}$ possess several nice properties (such as robustness to outliers), in practice they often overestimate the uncertainty about $\theta_0$, especially when the number of groups $m$ is large: indeed, if for example $\theta\in \mb R$ and Bernstein-von Mises theorem holds, then 
each $\Pi_{|G_j|}(\cdot|G_j)$ is ``approximately normal'' with covariance $\frac{m}{n}I^{-1}(\theta_0)$ (here, $I(\theta_0)$ is the Fisher information). 
However, the asymptotic covariance of the posterior distribution based on the whole sample is $\frac{1}{n}I^{-1}(\theta_0)$.
 
To overcome this difficulty, we propose a modification of our approach where the random measures $\Pi_{n}^{(j)}$ are replaced by the
\textit{stochastic approximations} $\Pi_{|G_j|,m}(\cdot|G_j)$,  $j=1,\ldots ,m$ of the full posterior distribution.
To this end, define the ``stochastic approximation'' based on the subsample $G_j$ as
\begin{align}
\label{eq:approx}
\Pi_{|G_j|,m}(B|G_j):= \frac{\int\limits_B\left(\prod_{i\in G_j} p_\theta(X_i)\right)^m d\Pi(\theta)}
{\int\limits_{\Theta}\left(\prod_{i\in G_j} p_\theta(X_i)\right)^m d\Pi(\theta)},
\end{align}
where we assume that $p_\theta^m(\cdot)$ is an integrable function for all $\theta$. 
In other words, $\Pi_{|G_j|,m}(\cdot|G_j)$ is obtained as a posterior distribution given that each data point from $G_j$ is observed $m$ times.
While each of $\Pi_{|G_j|,k}(\cdot|G_j)$ might underestimate uncertainly, the median $\hat \Pi^{\st}_{n,g}$ (or $\hat \Pi^{\st}_{n,0}$) of these random measures yields credible sets with much better coverage.
This approach shows good performance in numerical experiments. 
One of our main results (see section \ref{sec:approx}) provides a justification for this observation (albeit, under rather strong assumptions and for the parametric case). 


\subsection{Convergence of posterior distribution and robust Bayesian inference}
\label{sec:wasserstein}

In this subsection, we study the contraction and robustness properties of the median posterior.  

Our first result establishes the ``weak concentration'' property of the posterior distribution around the true parameter.
Let $\delta_{0}:=\delta_{\theta_0}$ be the Dirac measure supported on $\theta_0\in \Theta$.
Recall the following version of Theorem 2.1 in \cite{Ghosal2000Convergence-rat00} (we state the result for the Wasserstein distance $d_{W_{1,\rho}}(\Pi_n(\cdot|\m X_l),\delta_0)$ rather than the (closely related) contraction rate of the posterior distribution).
Here, the Wasserstein distance is evaluated with respect to the ``Hellinger metric'' $\rho(\cdot,\cdot)$ defined in (\ref{eq:metric}).
\begin{theorem}
\label{thm:1}
Let $\m X_l=\{X_1,\ldots,X_l\}$ be an i.i.d. sample from $P_0$.
Assume that $\eps_l>0$ and $\Theta_l\subset \Theta$ are such that for some constant $C>0$
\begin{align*}
&
(1) \ \text{the packing number satisfies }\log M(\eps_l,\Theta_l,\rho)\leq l\eps_l^2, \\
&
(2) \ \Pi(\Theta\setminus \Theta_l)\leq \exp(-l \eps_l^2(C+4)),\\
&
(3) \ \Pi\left(\theta: \ -P_0\left(\log\frac{p_\theta}{p_0}\right)\leq \eps_l^2, \ P_0\left(\log\frac{p_\theta}{p_0}\right)^2\leq \eps_l^2\right)
\geq \exp(-Cl\eps_l^2).
\end{align*}
Then there exists $R=R(C)$ and a universal constant $\tilde K$ such that
\begin{align}
\label{eq:wass0}
\Pr\Big(d_{W_{1,\rho}}(\delta_0,\Pi_l(\cdot|\m X_l))\geq R\eps_l+e^{-\tilde K l\eps_l^2}\Big)\leq \frac{1}{l\eps_l^2}+4e^{-\tilde K l\eps_l^2}.
\end{align}
\end{theorem}
\begin{proof}
The proof closely mimics the argument behind Theorem 2.1 in \cite{Ghosal2000Convergence-rat00}. 
Details are outlined in section \ref{sec:proof1}. 
\end{proof}

Conditions of Theorem \ref{thm:1} are standard assumptions guaranteeing that the resulting posterior distribution contracts to the true parameter $\theta_0$ at the rate $\eps_n$. 
Note that the bounds for the distance $d_{W_{1,\rho}}(\delta_0,\Pi_l(\cdot|\m X_l)$ slightly differ from the contraction rate itself: indeed, we have
\[
d_{W_{1,\rho}}(\delta_0,\Pi_l(\cdot|\m X_l))\leq \eps_l+
\int\limits_{h(P_\theta,P_0)\geq \eps_l}d\Pi_l(\cdot|\m X_l),
\]
hence to obtain the inequality $d_{W_{1,\rho}}(\delta_0,\Pi_l(\cdot|\m X_l))\lesssim \eps_l$, 
we usually require 
\[
\int\limits_{h(P_\theta,P_0)\geq \eps_l}d\Pi_l(\cdot|\m X_l)\lesssim \eps_l, 
\]
which adds an extra logarithmic factor in the parametric case. 

Combination of Theorems \ref{thm:1} and \ref{thm:main} immediately yields the corollary for $\hat \Pi_{n,0}$. 
Let $\mb H$ be the reproducing kernel Hilbert space with the reproducing kernel 
\[
k_H(\theta_1,\theta_2)=\frac{1}{2}\dotp{\sqrt{p_{\theta_1}}}{\sqrt{p_{\theta_2}}}_{L_2(dx)}.
\]
Let $f\in \mb H$ and note that, due to the reproducing property and Cauchy-Schwarz inequality, we have
\begin{align}
\label{eq:inclusion}
\nonumber
f(\theta_1)-f(\theta_2)&
=\dotp{f}{k_H(\cdot,\theta_1)-k_H(\cdot,\theta_2)}_\mb H \\
&\leq \left\|f\right\|_\mb H \big\|k_H(\cdot,\theta_1)-k_H(\cdot,\theta_2)\big\|_{\mb H}=
\left\|f\right\|_\mb H \,\rho(\theta_1,\theta_2).
\end{align}
Therefore, $\m F_k\subseteq \m F_L$ and $\|P-Q\|_{\m F_k}\leq \|P-Q\|_{\m F_L}$, where
$\m F_k$ and $\m F_L$ were defined in (\ref{eq:f_k}) and (\ref{eq:f_l}) respectively, and the underlying metric structure is given by $\rho$. 	
In particular, convergence with respect to $\|\cdot\|_{\m F_L}$ implies convergence with respect to $\|\cdot\|_{\m F_k}$.

\begin{corollary}
\label{cor:med1}
Let $X_1,\ldots,X_n$ be an i.i.d. sample from $P_0$, and assume that $\hat\Pi_{n,g}$ is defined with respect to the norm
$\|\cdot\|_{\m F_L}$ as in (\ref{eq:med_0}) above.
Set $l:=\lfloor n/m\rfloor$, assume that conditions of Theorem \ref{thm:1} hold, and, moreover, that $\eps_l$ satisfies
\[
\frac{1}{l\eps_l^2}+4e^{-(1+K/2)l\eps_l^2/2}<\frac 1 7.
\]
Then
\[
\Pr\Big(\l\|\delta_0-\hat\Pi_{n,g}\r\|_{\m F_{k_H}}\geq 1.52\l(R\eps_l+e^{-\tilde K l\eps_l^2}\r)\Big)\leq \left[e^{\psi\left(3/7,1/7\right)}\right]^{-m} < 1.27^{-m}.
\]
\end{corollary}
\begin{proof}
It is enough to apply part (a) of Theorem \ref{thm:main} with $\kappa=0$ to the independent random measures $\Pi_n(\cdot|G_j), \ j=1,\ldots,m$.
Note that the ``weak concentration'' assumption (\ref{eq:weak_con2}) is implied by (\ref{eq:wass0}).
\end{proof}
Once again, note the exponential improvement of concentration as compared to Theorem \ref{thm:1}. 
It is easy to see that a similar statement holds for the median $\hat\Pi_{n,0}(\cdot)$ defined in (\ref{eq:med_0}) (even for the stronger Wasserstein distance $d_{W_{1,\rho}}(\delta_0,\hat\Pi_{n,0})$), modulo changes in constants.

\begin{remark}
\label{remark3.6}
The case when the sample $\m X_n=\{X_1,\ldots, X_n\}$ contains $\lfloor\kappa m\rfloor$ outliers (which can be completely arbitrary vectors in $\mb R^D$) for some $\kappa<1/3$ can be handled similarly.
In most examples throughout the paper, we state the results for the case $\kappa=0$ for simplicity, keeping in mind that the generalization is a trivial corollary of Theorem \ref{thm:main}. 
For example, if we allow $\lfloor\kappa m\rfloor$ outliers in the setup of Corollary \ref{cor:med1}, the resulting bounds becomes 
\[
\Pr\Big(\l\|\delta_0-\hat\Pi_{n,g}\r\|_{\m F_{k_H}}\geq 1.52\l(R\eps_l+e^{-\tilde K l\eps_l^2}\r)\Big)\leq 
\left[e^{(1-\kappa)\psi\left(\frac{3/7-\kappa}{1-\kappa},1/7\right)}\right]^{-m}.
\]
\end{remark}

While the result of the previous statement is promising, numerical approximation and sampling from the ``robust posterior'' $\hat\Pi_{n,g}$ is often problematic due to the underlying geometry defined by the Hellinger metric, and the associated distance $\|\cdot\|_{\m F_{k_H}}$ is hard to estimate in practice. 
Our next goal is to derive similar guarantees for the M-posterior evaluated with respect to the computationally tractable family of distances discussed in section \ref{sec:distances} above. 

To transfer the conclusions of Theorem \ref{thm:1} and Corollary \ref{cor:med1} to the case of other kernels $k(\cdot,\cdot)$ and associated metrics $\rho_k(\cdot,\cdot)$, we need to guarantee the existence of tests versus the complements of the balls in these distances. 
Such tests can be obtained from comparison inequalities between distances. 

\begin{assumption}
\label{equiv}
There exists $\gamma>0$, $r(\theta_0)>0$ and $\tilde C(\theta_0)>0$ satisfying
\[
d(\theta,\theta_0)\geq \tilde C(\theta_0)\rho^\gamma_k(\theta,\theta_0) \text{ whenever } d(\theta,\theta_0)\leq r(\theta_0),
\]
where $d$ is the Hellinger distance or the Euclidean distance (in the parametric case). 
\end{assumption}
\begin{remark}
\label{rem:equiv}
When $d$ is the Euclidean distance, we will impose an additional mild assumption guaranteeing existence of test versus the complements of the balls (for the Hellinger distance, this is always true, see \cite{Ghosal2000Convergence-rat00}).
Namely, we will assume that for every $n$ and every pair $\theta_1,\theta_2\in \Theta$, there exists a test $\phi_n:=\phi_n(X_1,\ldots,X_n)$ such that for some $\gamma>0$ and a universal constant $K>0$ 
\begin{align}
&
\nonumber
\mb E_{P_{\theta_1}}\phi_n \leq e^{-Kn d^{2}(\theta_1,\theta_2)},\\
\sup_{d(\theta,\theta_2) < d(\theta_1,\theta_2)/2}
&
\label{eq:equiv}
\mb E_{P_{\theta}}\l( 1-\phi_n \r) \leq e^{-Kn d^{2}(\theta_1,\theta_2)}.
\end{align}

\end{remark}

Below, we provide several examples of kernels satisfying the stated assumption. 
\begin{example}[Exponential families]

Let $\{P_\theta, \ \theta\in \Theta\subseteq \mb R^p\}$ be of the form
\[
\frac{dP_\theta}{dx}(x):=p_\theta(x)=\exp\Big(\dotp{T(x)}{\Theta}_{\mb R^p}-G(\theta)+q(x)\Big),
\]
where $\dotp{\cdot}{\cdot}_{\mb R^p}$ is the standard Euclidean dot product.
Then the Hellinger distance can be expressed as \citep{Nielsen2009Statistical-exp00}
\[
h^2(P_{\theta_1},P_{\theta_2})=1-\exp\Big(-\frac 1 2 \Big(G(\theta_1)+G(\theta_2)-2G\Big(\frac{\theta_1+\theta_2}{2}\Big)\Big)\Big).
\]
If $G(\theta)$ is convex and its Hessian $D^2 G(\theta)$ satisfies $D^2 G(\theta)\succeq A$ uniformly for all $\theta\in \Theta$ and some symmetric positive definite operator $A:\mb R^p\mapsto \mb R^p$ , then
\begin{align*}
h^2(P_{\theta_1},P_{\theta_2})\geq 1-\exp\bigg(-\frac{1}{8}(\theta_1-\theta_2)^T A(\theta_1-\theta_2)\bigg),
\end{align*}
hence assumption \ref{equiv} holds with $d$ being the Hellinger distance, $\gamma=1$, $\tilde C=\frac{1}{\sqrt 2}$ and $r(\theta_0)\equiv 1$ for
\[
k(\theta_1,\theta_2):=\exp\l(-\frac{1}{8}(\theta_1-\theta_2)^T A(\theta_1-\theta_2)\r).
\]
\end{example}

For finite-dimensional models, we will be especially interested in kernels $k(\cdot,\cdot)$ such that the associated metric $\rho_k(\cdot,\cdot)$ is bounded by the Euclidean distance. 
The following proposition gives a sufficient condition for this to hold. 

Assume that kernel $k(\cdot,\cdot)$ satisfies conditions of Proposition \ref{prop:kernel} (in particular, $k$ is characteristic).
Recall that by Bochner's theorem, there exists a finite nonnegative Borel measure $\nu$ such that
$k(\theta)=\int\limits_{\mb R^p}e^{i\dotp{x}{\theta}}d\nu(x)$.

\begin{proposition}
\label{comparison}
Assume that $\int\limits_{\mb R^p}\|x\|_2^2 d\nu(x)<\infty$.
Then there exists $D_k>0$ depending only on $k$ such that for all $\theta_1,\theta_2$,
\[
\rho_k(\theta_1,\theta_2)\leq D_k\l\|\theta_1-\theta_2\r\|_2.
\]
\end{proposition}
\begin{proof}
For all $z\in \mb R$, $|e^{iz}-1-iz|\leq \frac{|z|^2}{2}$, implying that
\begin{align*}
\rho^2_k(\theta_1,\theta_2)&=\|k(\cdot,\theta_1)-k(\cdot,\theta_2)\|_{\mb H}^2\\
&=2k(0)-2k(\theta_1-\theta_2)=
2\int\limits_{\mb R^p}(1-e^{i\dotp{x}{\theta_1-\theta_2}})d\nu(x) \\
&\leq\int\limits_{\mb R^p}\dotp{x}{\theta_1-\theta_2}_{\mb R^p}^2 d\nu(x)\leq \|\theta_1-\theta_2\|_2^2\int\limits_{\mb R^p}\|x\|_2^2 d\nu(x).
\end{align*}
\end{proof} 
Moreover, the result of the previous proposition clearly remains valid for kernels of the form 
\begin{align}
\label{spkernel}
\tilde k(\theta_1,\theta_2) = k(\theta_1 - \theta_2) + c\dotp{\theta_1}{\theta_2}_{\mb R^p},
\end{align}
where $c>0$ and $k$ satisfies the assumptions of proposition \ref{comparison}. 
For such a kernel, we have the obvious lower bound 
$
\rho_{\tilde k}(\theta_1,\theta_2)\geq \sqrt{c}\l\|\theta_1-\theta_2\r\|_2,
$
hence $\rho_{\tilde k}$ is equivalent (in the strong sense) to the Euclidean distance. 

We are ready to state our main result for convergence with respect to the RKHS-induced distance $\|\cdot\|_{\m F_k}$.
\begin{theorem}
\label{thm:2}
Assume that conditions of Theorem \ref{thm:1} hold with $\rho$ being the Hellinger or the Euclidean distance, and that assumption \ref{equiv} is satisfied. 
In addition, let prior $\Pi$ be such that 
\[
D^W_k:=\int_\Theta \rho^W_k(\theta,\theta_0)d\Pi(\theta)<\infty
\]
for a sufficiently large $W$.\footnote{It follows from the proof that $W=\frac{4}{3}+\frac{4+2C}{3\tilde K}$ is sufficient, with $C$ and $\tilde K$ being the constants from the statement of Theorem \ref{thm:1}.}
Then there exists a sufficiently large 
$R=R(\theta_0,\gamma)>0$ and an absolute constant $\tilde K$ such that
\begin{align}
\label{eq:wass3}
& 
\Pr\Big(\l\|\delta_0-\Pi_l(\cdot|\m X_l)\r\|_{\m F_k}\geq R\eps_l^{1/\gamma}+D_k e^{-\tilde K l\eps_l^2/2}\Big)
\leq \frac{1}{l\eps_l^2}+4e^{-\tilde Kl\eps_l^2}.
\end{align}
\end{theorem}
\begin{proof}
The result essentially follows from the combination of Theorem \ref{thm:1} and assumption \ref{equiv}, see section \ref{proof:3.8} in the appendix for details.
\end{proof}

Theorem \ref{thm:2} yields the ``weak'' estimate that is needed to obtain the stronger bound for the M-Posterior distribution $\hat \Pi_{n,g}$.
This is summarized in the following corollary:

\begin{corollary}
\label{cor:med2}
Let $X_1,\ldots,X_n$ be an i.i.d. sample from $P_0$, and assume that $\hat\Pi_{n,g}$ is defined with respect to the distance $\|\cdot\|_{\m F_k}$ as in (\ref{eq:med_0}) above.
Let $l:=\lfloor n/m\rfloor$.
Assume that conditions of Theorem \ref{thm:2} hold, and, moreover, $\eps_l$ is such that
\[
\frac{1}{l\eps_l^2}+4e^{-\tilde Kl\eps_l^2}<\frac 1 7.
\]
Then
\begin{align}
\label{eq:conc2}
&
\Pr\Big(\big\| \delta_0-\hat\Pi_{n,g} \big\|_{\m F_k}\geq 1.52\l(R\eps_l^{1/\gamma}+D_k e^{-\tilde K l\eps_l^2/2}\r)   \Big)\leq 1.27^{-m}.
\end{align}
\end{corollary}

\begin{proof}
It is enough to apply parts (a) and (b) of Theorem \ref{thm:main} with $\kappa=0$ to the independent random measures 
$\Pi_{|G_j|}(\cdot|G_j), \ j=1,\ldots,m$.
Note that the ``weak concentration'' assumption (\ref{eq:weak_con1}) is implied by (\ref{eq:wass3}).
\end{proof}

Note that if $\Theta\subseteq \mb R^p$ and kernel $k(\cdot,\cdot)$ is of the form (\ref{spkernel}), the previous corollary together with proposition \ref{prop:mean} imply that 
\[
\Pr\Big(\big\|  \theta_\ast - \delta_0\big\|_{2}\geq \frac{1.52 R}{c}\eps_l\Big)\leq 1.27^{-m}, 
\]
where $\theta_\ast = \int_\Theta \theta d\hat\Pi_{n,g}(\theta)$ is the mean of $\hat\Pi_{n,g}$. 
In other words, this shows that the M-posterior mean is the ``robust'' estimator of $\theta_0$.

\subsection{Bayesian inference based on stochastic approximation of the posterior distribution}
\label{sec:approx}

As we have already mentioned in section \ref{sec:construct}, 
when the number of disjoint subgroups $m$ is large, the resulting M-Posterior distribution is ``too flat'', which results in large credible sets and overestimation of uncertainty. 
Clearly, the source of the problem is the fact that each individual random measure $\Pi_{|G_j|}(\cdot|G_j), \ j=1,\ldots, m$ is based on sample of size $l\simeq \frac{n}{m}$ which can be much smaller  than $n$. 

One way to reduce the variance of each subset posterior distribution is to repeat each observation in $G_j$ $m$ times (although other alternatives, such as bootstrap, are possible), $\tilde G_j=\underbrace{\{G_j,\ldots,G_j\}}_{m \text{ times}}$.
Formal application of the Bayes rule in this situation yields a collection of new measures on the parameter space:
\[
\Pi_{|G_j|,m}(B|G_j) := \frac{\int\limits_B\left(\prod_{i\in G_j} p_\theta(X_i)\right)^m d\Pi(\theta)}
{\int\limits_{\Theta}\left(\prod_{i\in G_j} p_\theta(X_i)\right)^m d\Pi(\theta)},
\]
where we have assumed that $p_\theta(\cdot)$ is integrable. 
Here, $\left(\prod_{i\in G_j} p_\theta(X_i)\right)^m$ can be viewed as an approximation of the full data likelihood.
We call the random measure $\Pi_{|G_j|,m}(\cdot|G_j)$ the \textit{$j$-th stochastic approximation} to the full posterior distribution. 

Of course, such a ``correction'' negatively affects coverage properties of the credible sets associated with each measure $\Pi_{|G_j|}(\cdot|G_j)$. 
However, taking the median of stochastic approximations yields improved coverage of the resulting M-posterior distribution. 
The main goal of this section is to establish an asymptotic statement in spirit of a Bernstein-von Mises theorem for the M-posterior based on stochastic approximations $\Pi_{|G_j|,m}(B|G_j), \ j=1,\ldots, m$. 


We will start by showing that under certain assumptions the upper bounds for the convergence rates of 
$\Pi_{|G_j|,m}(\cdot|G_j)$ towards $\delta_0$ are the same as for $\Pi_{|G_j|}(\cdot|G_j)$, the ``standard'' posterior distribution given $G_j$.

For $A\subseteq \Theta$, let $N_{[\,]}(u,A,d)$ be the bracketing number of $\{p_\theta, \ \theta\in A\}$ with respect to the distance 
$d(l,u):=\int\limits_{\mb R^D}\l(\sqrt{l(x)}-\sqrt{u(x)}\r)^2 dx$,
and let 
\[
H_{[\,]}(u; A):=\log N_{[\,]}(u,A,d)
\] 
be the \textit{bracketing entropy}. 
In what follows, $B(\theta_0,r):=\{\theta\in \Theta: \ h(P_\theta, P_{\theta_0})\leq r\}$ 
denotes the ``Hellinger ball'' of radius $r$ centered at $\theta_0$.

\begin{theorem}[\cite{wong1995probability}, Theorem 1]
\label{th:wong}
There exist constants $c_j, \ j=1,\ldots ,4$ and $\zeta>0$ such that if
\[
\int\limits_{\zeta^2/2^8}^{\sqrt 2 \zeta}H_{[\,]}^{1/2}\l(u/c_3;B(\theta_0,\zeta\sqrt{2})\r)du\leq c_4 \sqrt l \zeta^2,
\]
then
\[
P\left(\sup\limits_{\theta: h(P_\theta,P_0)\geq\zeta}\prod\limits_{j=1}^l \frac{p_\theta}{p_0}(X_j)\geq e^{-c_1 l \zeta^2}\right)\leq 4e^{-c_2 l\zeta^2}.
\]
In particular, one can choose $c_1=1/24, \ c_2=(4/27)(1/1926), \ c_3=10$ and $c_4=(2/3)^{5/2}/512$.
\end{theorem}
In ``typical'' parametric problems ($\Theta\subseteq \mb R^p$), the bracketing entropy can be bounded as 
$H_{[\,]}(u;B(\theta_0,r))\leq C_1\log(C_2r/u)$, 
whence the minimal $\zeta$ that satisfies conditions of Theorem \ref{th:wong} is of order $\zeta \simeq\sqrt{\frac{1}{l}}$. 
In particular, it is easy to check (e.g., using Theorem 2.7.11 in \cite{Vaart1996Weak-convergenc00}) that this is the case when
\begin{enumerate}[(a)]
\item there exists $r_0>0$ such that
\[
h\l(P_\theta,P_{\theta_0}\r)\geq K_1\|\theta-\theta_0\|_2
\] 
whenever $h\l(P_\theta,P_{\theta_0}\r)\leq r_0$, and
\item 
there exists $\alpha>0$ such that for $\theta_1,\theta_2\in B(\theta_0,r_0)$, 
\[
\l|p_{\theta_1}(x)-p_{\theta_2}(x)\r|\leq F(x)\l\|\theta_1-\theta_2\r\|^{\alpha}_2\] 
with $\int_{\mb R^D}F(x)dx<\infty $.
\end{enumerate}
Application of theorem \ref{th:wong} to the analysis of ``stochastic approximations'' yields the following result. 

\begin{theorem}
\label{thm:dist2} \text{ }\\
Let $\eps_l>0$ be such that conditions of Theorem \ref{th:wong} hold with $\zeta:=\eps_l$, and
\begin{enumerate}
\item[(a)]
for some $C>0$
\begin{align*}
&
\Pi\left(\theta: \ -P_0\left(\log\frac{p_\theta}{p_0}\right)\leq \eps_l^2, \ P_0\left(\log\frac{p_\theta}{p_0}\right)^2\leq \eps_l^2\right)
\geq \exp(-Cl\eps_l^2),\\
\end{align*}
\item[(b)] $k$ is a positive-definite kernel that satisfies assumption \ref{equiv} for the Hellinger distance for some 
$\tilde C(\theta_0)$ and $\gamma>0$.
\end{enumerate}
Then there exists $\tilde R=\tilde R(C,\tilde C,\gamma)>0$ such that
\begin{align*}
&
\Pr\left(\l\|\delta_0-\Pi_{l,m}(\cdot|\m X_l)\r\|_{\m F_k}\geq \tilde R\eps^{1/\gamma}_l+e^{-ml\eps_l^2}\right)\leq
\frac{1}{ l\eps_l^2}+4e^{-c_2 \tilde C^2 \tilde R^{2\gamma} l\eps_l^2}.
\end{align*}

\end{theorem}
\begin{proof}
See section \ref{proof:dist2} in the appendix.

\end{proof}
\begin{remark}
Note that for the kernel $k(\cdot,\cdot)$ of the form (\ref{spkernel}), assumption \ref{equiv} reduces to the inequality between the Hellinger and Euclidean distances. 
\end{remark}

As before, Theorem \ref{thm:main} combined with the ``weak concentration'' inequality of Theorem \ref{thm:dist2} gives stronger guarantees for the median $\hat \Pi^{\st}_{n,g}$ (or its alternative $\hat \Pi^{\st}_{n,0}$) of $\Pi_{|G_1|,m}(\cdot|G_1),\ldots,\Pi_{|G_m|,m}(\cdot|G_m)$. 
Exact statement is very similar in spirit to Corollary \ref{cor:med2}.  

Our next goal is to obtain the result describing the asymptotic behavior of the M-posterior distribution $\hat \Pi^{\st}_{n,0}$ in the parametric case.  
We start with a result that addresses each individual stochastic approximation $\Pi_{|G_j|,m}(\cdot|G_j), \ j=1,\ldots,m$. 
Assume that $\Theta\subseteq \mb R^p$ has non-empty interior. 
For $\theta\in \Theta$, let 
\[
I(\theta):=\mb E_{\theta_0} \l[\frac{\partial}{\partial \theta}\log p_\theta(X) \l(\frac{\partial}{\partial \theta}\log p_\theta(X)\r)^T\r]
\]
be the Fisher information matrix (we are assuming that it is well-defined). 
We will say that the family $\{P_\theta, \ \theta\in \Theta\}$ is \textit{differentiable in quadratic mean} (see Chapter 7 in\cite{van2000asymptotic} for details) if there exists 
$\dot\ell_{\theta_0}:\mb R^D\mapsto \mb R^p$ such that 
\[
\int_{\mb R^D}\l(\sqrt{p_{\theta_0+h}}-\sqrt{p_{\theta_0}}-\frac{1}{2}h^T\dot\ell_{\theta_0}\sqrt{p_{\theta_0}}\r)^2=o(\|h\|_2^2)
\]
as $h\to 0$; usually, $\dot\ell_{\theta}(x)=\frac{\partial}{\partial \theta}\log p_\theta(x)$.
Next, define 
\[
\Delta_{l,\theta_0}:=\frac{1}{\sqrt{l}}\sum_{j=1}^l I^{-1}(\theta_0)\dot\ell_{\theta_0}(X_j).
\]
We will first state a preliminary result for each individual ``subset posterior'' distribution:
\begin{proposition}
\label{thm:bvm}
Let $X_1,\ldots,X_l$ be an i.i.d. sample from $P_{\theta_0}$ for some $\theta_0$ in the interior of $\Theta$. 
Assume that 
\begin{enumerate}[(a)]
\item the family $\{P_\theta, \ \theta\in \Theta\}$ is differentiable in quadratic mean;
\item the prior $\Pi$ has a density (with respect to the Lebesgue measure) that is continuous and positive in the neighborhood of $\theta_0$;
\item conditions of Theorem \ref{th:wong} hold with $\zeta=\frac{C}{\sqrt l}$ for some $C>0$ and $l$ large enough. 
\end{enumerate}
Then for any integer $m\geq 1$,
\[
\l\|\Pi_{l,m}(\cdot|X_1,\ldots,X_l) - N\l(\theta_0+\frac{\Delta_{l,\theta_0}}{\sqrt{l}},\frac{1}{l\cdot m}I^{-1}(\theta_0)\r)\r\|_{\rm TV}\to 0 
\]
in $P_{\theta_0}$-probability as $l\to \infty$.
\end{proposition}
\begin{proof}
The proof follows standard steps (e.g., Theorem 10.1 in \cite{van2000asymptotic}), where the existence of tests is substituted by the inequality of Theorem \ref{th:wong}. 
See section \ref{proof:3.13} in the appendix for more details.
\end{proof}
The implication of this result for the M-posterior is the following: if $k$ is the kernel of type \eq{spkernel}, for sufficiently regular parametric families (differentiable in quadratic mean, with ``well-behaved'' bracketing numbers, satisfying assumption \ref{equiv} for the Euclidean distance with $\gamma=1$) and regular priors, then
\begin{enumerate}[(a)]
\item the M-posterior is well approximated by a normal distribution centered at the ``robust'' estimator $\theta^\ast$ of unknown $\theta_0$;
\item the estimator $\theta^\ast$ is a center of the confidence set of level $1.15^{-m}$ and diameter of order $\sqrt{\frac{m}{n}}$ (same as we would expect for this level for the usual posterior distribution - however, the bound for the M-posterior holds for finite sample sizes).
\end{enumerate} 
This is formalized below:
\begin{theorem}
\begin{enumerate}[(a)]
\item
Let $k$ be the kernel of type \eq{spkernel}, and suppose that the assumptions of Proposition \ref{thm:bvm} hold. 
Moreover, let the prior $\Pi$ be such that $\int_{\mb R^p} \|\theta\|_2^2 d\Pi(\theta)<\infty$. 
Then for any fixed $m\geq 1$,
\[
\l\|\hat \Pi^{\st}_{n,0} - N\l(\theta^\ast,\frac{1}{n}I^{-1}(\theta_0)\r)\r\|_{\rm TV}\to 0 \text{ as } n\to\infty,
\]
in $P_{\theta_0}$-probability when $n\to\infty$, 
where $\theta^\ast$ is the mean of $\hat\Pi^{\st}_{n,0}$.
\item Assume that conditions (a), (b) of Theorem \ref{thm:dist2} hold with 
\[
\eps_l \gtrsim \frac{1}{\sqrt{l}}\simeq \sqrt{\frac{m}{n}}
\] 
and $\gamma=1$. 
Then for all $n\geq n_0$ and $\bar R$ large enough,
\[
\Pr\l(\l \|\theta^\ast - \theta_0 \r\|_2\geq \bar R\l(\eps_l+e^{-ml\eps_l^2}\r)\r)\leq 1.15^{-m}.
\]
\end{enumerate} 
\end{theorem}
\begin{proof}
(a) It is easy to see that convergence in total variation norm, together with an assumption that the prior distribution satisfies 
\[
\int_{\mb R^p} \|\theta\|_2^2 d\Pi(\theta)<\infty,
\]
implies that the expectations converge in $P_{\theta_0}$-probability as well:
\[
\l\|\int_\Theta \theta \l(d\Pi_{l,m}(\theta|\m X_l) - d N\l(\theta_0+\frac{\Delta_{l,\theta_0}}{\sqrt{l}},\frac{1}{l\cdot m}I^{-1}(\theta_0)\r)(\theta)\r)  \r\|_2\to 0 \text{ as } l\to \infty.
\]
Together with an observation that the total variation distance between $N(\mu_1,\Sigma)$ and $N(\mu_2,\Sigma)$ is bounded by the multiple of $\l \|\mu_1-\mu_2 \r\|_2$, it implies that we can replace $\theta_0+\frac{\Delta_{l,\theta_0}}{\sqrt{l}}$ by the mean 
\[
\bar\theta_{l,m}(X_1,\ldots,X_l):=\int_\Theta  \theta d\Pi_{l,m}(\theta|\m X_l),
\]
in other words, the conclusion of Proposition \ref{thm:bvm} can be stated as
\[
\l\|\Pi_{l,m}(\cdot|\m X_l) - N\l(\bar\theta_{l,m},\frac{1}{l\cdot m}I^{-1}(\theta_0)\r)\r\|_{\rm TV}\to 0 \text{ as } l\to \infty,
\]
in $P_{\theta_0}$-probability. 
Now assume that $m=\lfloor\frac{n}{l}\rfloor$ is fixed, and let $n,l\to\infty$. 
As before, let $G_1,\ldots,G_m$ be disjoint groups of i.i.d. observations from $P_{\theta_0}$ of cardinality $l$ each. 
Recall that, by the definition (\ref{eq:nemir0}) of $\med_0(\cdot)$, $\Pi^{\st}_{n,0}=\Pi_{l,m}(\cdot|\m X_{l_\ast})$ for some $l_\ast\leq m$, and 
$\theta^\ast:=\theta_{l_\ast,m}$ is the mean of $\Pi^{\st}_{n,0}$. 
Clearly, we have
\begin{align}
\nonumber
&
\l\| \Pi^{\st}_{n,0} - N\l(\theta^\ast,\frac{1}{l\cdot m}I^{-1}(\theta_0)\r)\r\|_{\rm TV}\leq \\
&
\max_{j=1,\ldots, m}  \l\| \Pi_{l,m}(\cdot|G_j) - N\l(\bar\theta_{l,m}(G_j),\frac{1}{l\cdot m}I^{-1}(\theta_0)\r)\r\|_{\rm TV}\to 0 \text{ as } n\to\infty.
\end{align}
		
(b) 
Let $\eps_l\geq C\sqrt{\frac{1}{l}}$ where $C$ large enough so that 
\[
\frac{1}{ l\eps_l^2}+4e^{-c_2 \tilde C^2 \tilde R^{2} l\eps_l^2} \leq \frac{1}{4},
\]
where $c_2$, $\tilde R$ are the same as in Theorem \ref{thm:dist2}. 

Applying Theorem \ref{thm:dist2}, we get 
\[
\Pr\left(\l\|\delta_{\theta_0}-\Pi_{l,m}(\cdot|\m X_l)\r\|_{\m F_k}\geq \tilde R\eps_l+e^{-ml\eps_l^2}\right)\leq \frac{1}{4}.
\]
By part (b) of Theorem \ref{thm:main},
\[
\l\| \hat \Pi^{\st}_{n,0} - \delta_{\theta_0} \r\|_{\m F_{\tilde k}} \leq 3\l( \tilde R\eps_l+e^{-ml\eps_l^2}\r)
\] 
with probability $\geq 1-1.15^{-m}$. 
Since kernel $\tilde k$ is of the type (\ref{spkernel}), proposition (\ref{prop:mean}) implies that 
\begin{align}
\label{eq:bvm10}
&
\l\|  \theta^\ast - \theta_0 \r\|_2 \leq \l\| \hat \Pi^{\st}_{n,0} - \delta_{\theta_0} \r\|_{\m F_{\tilde k}},
\end{align}
and the result follows.

\end{proof}

In particular, for $m=A\log(n)$ and $\eps_l\simeq\sqrt{\frac{m}{n}}$, we obtain the bound
\[
\Pr\l(\|\theta^\ast - \theta_0\|_2\geq \bar R \sqrt{\frac{A\log n}{n}} \r)\leq n^{-A}.
\]
for some constant $\bar R$ \textit{independent} of $m$. 
Note that $\theta^\ast$ itself depends on $m$, hence this bound is \textit{not uniform}, and holds only for a given confidence level $1-n^{-A}$. 

It is convenient to interpret this (informally) in terms of the credible sets: to obtain the credible set with ``frequentist'' coverage level $\geq 1-n^{-A}$, pick $m=A\log n$ and use the $(1-n^{-A})$ - credible set of the M-posterior $\hat \Pi^{\st}_{n,0}$.

\section{Numerical algorithms and examples}
\label{sec:ex}

In this section, we consider examples and applications in which comparisons are made for the inference based on the usual posterior distribution and on the M-Posterior. One of the well-known and computationally efficient ways to find the geometric median in Hilbert spaces is the famous \emph{Weiszfeld's algorithm} (introduced in \cite{Weiszfeld1936Sur-un-probleme00}). Details of implementation are described in Algorithms \ref{algo:median} and \ref{algo:mpost}.  Algorithm \ref{algo:median} is a particular case of Weiszfeld's algorithm applied to subset posterior distributions and distance $\|\cdot\|_{\m F_k}$, while Algorithm \ref{algo:mpost} shows how to obtain an approximation to M-Posterior given the samples from $\Pi_{n,m}(\cdot|G_j), \ j=1\ldots m$. Note that the subset posteriors $\Pi_{n,m}(\cdot|G_j)$ whose ``weights'' $w_{\ast,j}$ in the expression of the M-Posterior are small (in our case, smaller than $1/(2m)$) are excluded from the analysis. Our extensive simulations show the empirical evidence in favor of this additional thresholding step. 

Detailed discussion of convergence rates and acceleration techniques for Weiszfeld's method from the viewpoint of modern optimization can be found in \cite{beck2013weiszfeld}. For alternative approaches and extensions of Weiszfeld's algorithm, see \cite{bose2003fast}, \cite{ostresh1978convergence}, \cite{overton1983quadratically}, \cite{chandrasekaran1990algebraic}, \cite{cardot2012recursive}, \cite{cardot2013efficient}, among other works. 

In all numerical simulations below, we use ``stochastic approximations'' and the corresponding median measure $\hat \Pi^{\st}_{n,g}$, unless noted otherwise.

\begin{algorithm}[tb]
\footnotesize
   \caption{Evaluating the geometric median of probability distributions via Weiszfeld's algorithm}
   \label{algo:median}
\begin{algorithmic}
   \STATE {\bfseries Input:} 
   \vskip -0.1in
    \begin{enumerate}
      \itemsep0em 
    \item Discrete measures $Q_1,\ldots,Q_m$;  \vskip -0.1in
    \item The kernel $k(\cdot,\cdot):\mb R^p\times \mb R^p\mapsto \mb R$; 
    \item Threshold $\eps>0$;
    \end{enumerate}
    \STATE \textbf{Initialize:}
    \begin{enumerate}
    \item  Set $w_j^{(0)}:=\frac{1}{m}$, $j=1\ldots m$;   
    \item  Set $Q_\ast^{(0)}:=\frac 1 m\sum\limits_{j=1}^m Q_j$;     
    \end{enumerate}
    \STATE 
  \REPEAT
  \STATE Starting from $t=0$, for each $j = 1, \ldots, m$:
  \begin{enumerate}

  \item  Update 
    $w^{(t+1)}_{j} = \frac{\|Q_\ast^{(t)}-Q_j\|^{-1}_{\m F_k}}{\sum\limits_{i=1}^m \|Q_\ast^{(t)}-Q_i\|^{-1}_{\m F_k}}$; (apply (\ref{eq:discrete}) to evaluate 
$\|Q_\ast^{(t)}-Q_i\|_{\m F_k}$); 
     \item Update $Q_\ast^{(t+1)}=\sum\limits_{j=1}^m w_j^{(t+1)}Q_j$; 
     \end{enumerate}
   \UNTIL{$\|Q_\ast^{(t+1)}-Q_\ast^{(t)}\|_{\m F_k}\leq \eps$};
   \STATE \textbf{Return:} $w_\ast := (w_1^{(t+1)}, \ldots, w_m^{(t+1)})$. 
\end{algorithmic}
\end{algorithm}
\normalsize

\begin{algorithm}[tb]
  \footnotesize
  \caption{Approximating the M-Posterior distribution}
   \label{algo:mpost}
   \begin{algorithmic}
   \STATE {\bfseries Input:} 
    \begin{enumerate}
    \item Samples $\{Z_{j,i}\}_{i=1}^{S_j}\sim \Pi_{n,m}(\cdot|G_j), \ j=1\ldots m$ (see equation (\ref{eq:approx}));
    \end{enumerate}
    \STATE \textbf{Do:}
    \begin{enumerate}
    \item $Q_j:=\frac{1}{S_j}\sum\limits_{i=1}^{S_j}\delta_{Z_{j,i}}$, $j=1\ldots m$ - empirical approximations of $ \Pi_{n,m}(\cdot|G_j)$.
    \item Apply Algorithm \ref{algo:median} to $Q_1,\ldots,Q_m$; return $w_\ast=(w_{\ast,1}\ldots w_{\ast,m})$;
    \item For $j=1,\ldots, m$, set $\bar w_j:=w_{\ast,j}I\{w_{\ast,j}\geq \frac{1}{2m}\}$; define 
    $\hat w_j^\ast:=\bar w_j/\sum_{i=1}^m \bar w_i$.
    \end{enumerate}
  
   \STATE \textbf{Return:} $\hat \Pi_{n,g}^{\st}:=\sum_{i=1}^m \hat w_i^\ast Q_i$.
\end{algorithmic}
\end{algorithm}
\normalsize

Before presenting the results of numerical analysis, let us remark on two important computational aspects. 
\begin{remark}
\label{rmk:automatic_m}
\item[(a)] 
The number of subsets $m$ appears as a ``free parameter'' entering the theoretical guarantees for M-Posterior. 
One interpretation of $m$ (in terms of the credible sets) is given in the end of section \ref{sec:approx}. 
Our results also imply that partitioning the data into $m=2k+1$ subsets guarantees robustness to the presence of $k$ outliers of arbitrary nature. 
 
In many applications, $m$ is dictated by the sample size and computational resources (e.g., the number of available machines). 
In section \ref{section:selection} of the appendix, we describe a heuristic approach to selection of $m$ that shows good practical performance. 
As a rule of a thumb, we recommend choosing $m\lesssim \sqrt n$ as larger values of $m$ lead to an M-posterior that overestimates uncertainty. This heuristic is supported by the numerical results presented below. 

\item[(b)]
It is easy to get a general idea regarding the potential improvement in computational time complexity achieved by the M-Posterior. 
Given the data set $\m X_n=\{X_1,\ldots,X_n\}$ of size $n$, let $t(n)$ be the running time of the algorithm (e.g., MCMC) that outputs a single observation from the posterior distribution $\Pi_n(\cdot|\m X_n)$. If the goal is to obtain $S$ samples from the posterior, then the total running time is $O\l(S\cdot t(n)\r)$. Let us compare this time with the running time needed to obtain $S$ samples from the $M$-posterior given that the algorithm is running on $m$ machines \textit{in parallel}. 
In this case, we need to generate $O\l(S\r)$ samples from each of $m$ subset posteriors, which is done in time $O\l(S \cdot t\l(\frac{n}{m}\r)\r)$, where $S$ is typically large and $m \ll n$. According to Theorem 7.1 in \cite{beck2013weiszfeld}, Weiszfeld's algorithm approximates the M-Posterior to degree of accuracy $\eps$ in at most $O(1/\eps)$ steps, and each of these steps has complexity $O(S^2)$ (which follows from (\ref{eq:discrete})), so that the total running time is 
\begin{align}
  \label{eq:time-comp}
  O\l(S \cdot t\l(\frac{n}{m}\r)+\frac{S^2}{\eps}\r). 
\end{align}
The term $\frac{S^2}{\eps}$ can be refined in several ways via application of more advanced optimization techniques (see the aforementioned references). If, for example, $t(n)\simeq n^r$ for some $r\geq 1$, then $\frac{S}{m}\cdot t\l(\frac{n}{m}\r)\simeq \frac{1}{m^{1+r}}S n^r$ which should be compared to $S \cdot n^r$ required by the standard approach. 

To give a specific example, consider an application of \eqref{eq:time-comp} in the context of Gaussian process (GP) regression. 
If $n$ is the number of training samples, then GP regression
has $O(n^3)+O(S n^2)$ asymptotic time complexity to obtain $S$ samples from the posterior distribution of GP \citep[Algorithm 2.1]{RasWil06}. Assuming we have access to $m$ machines, the time complexity to obtain $S$ samples from M-Posterior in GP regression is $O \left( \left( \tfrac{n}{m} \right)^3  + S \left( \tfrac{n}{m} \right)^2 + \tfrac{S^2}{\eps} \right)$. 
If for example $S = cn$ for some $c > 0$ and $m^2< n\eps$, we get $O(m^2)$ improvement in running time.  


\item[(c)] In many cases, replacing the ``subset posterior'' by the stochastic approximation does not result in increased sampling complexity: indeed, the log-likelihood in the sampling algorithm for the subset posterior is simply multiplied by $m$ to obtain the sampler for the stochastic approximation. 
We have included the description of a modified Dirichlet mixture model in section \ref{section:dirichlet} of the appendix as an illustration. 

\end{remark}

\subsection{Numerical analysis: simulated data}
\label{sec:simul-data-analys}



\begin{figure}[ht]
\begin{center}
\centerline{\includegraphics[scale=0.2]{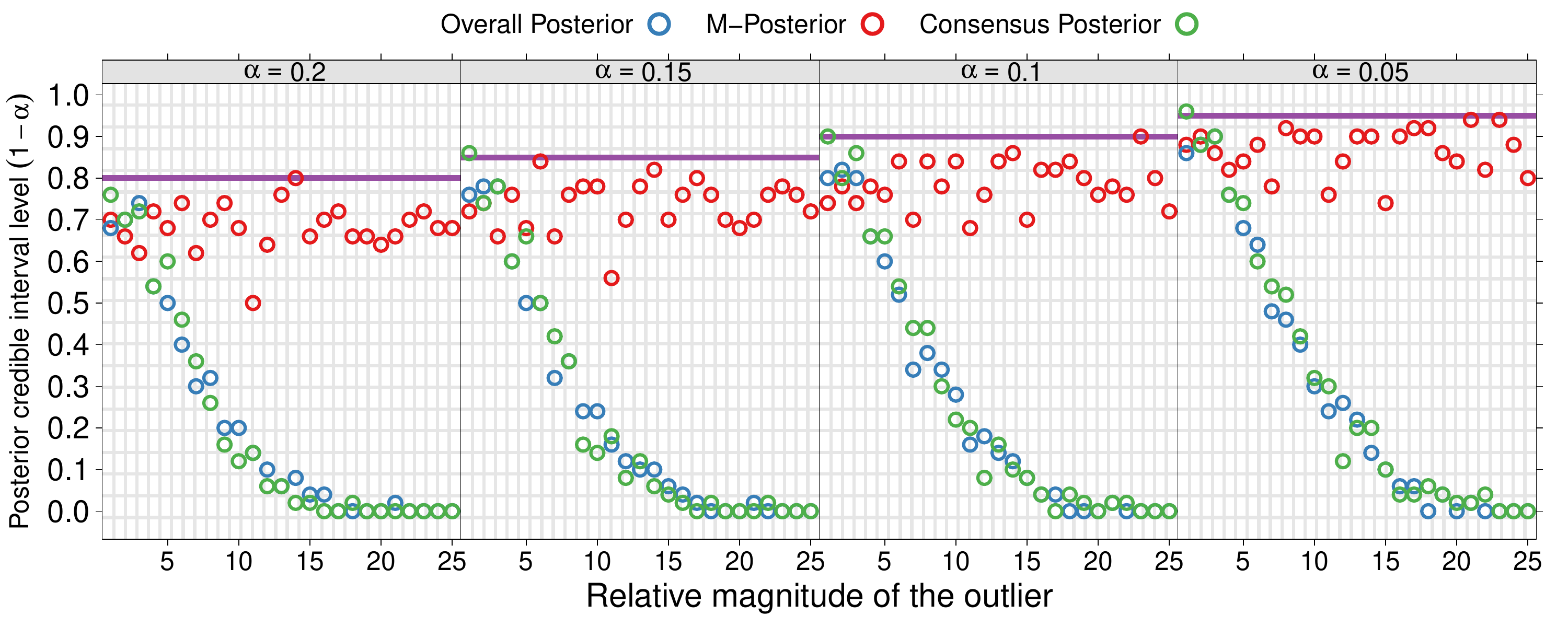}}
\caption{Effect of the outlier on empirical coverage of (1-$\alpha$)100\% credible intervals (CIs). The x-axis represents the outlier magnitude. The y-axis represents the fraction of times the CIs include the true mean over 50 replications. The horizontal lines (in violet) show the theoretical frequentist coverage.}
\label{norm-eg1}
\end{center}
\end{figure} 

This section demonstrates the effect of magnitude of an outlier on the posterior distribution of the mean parameter $\mu$. We empirically show that M-Posterior of $\mu$ is a robust alternative to the overall posterior. To this end, we used the simplest univariate Gaussian model $\{P_\mu=\m N(\mu,1), \ \mu\in \mb R\}$. 
 
We simulated $25$ data sets containing $100$ observations each. Each data set $\xb_i = (x_{i,1}, \ldots, x_{i,100})$ contained 99 independent observations from the standard Gaussian distribution ($x_{i,j} \sim \Ncal(0, 1)$ for $i = 1, \ldots, 25$ and $j = 1, \ldots, 99$). The last entry in each data set $x_{i,100}$ was an outlier, and its value  increased linearly for $i = 1, \ldots, 25$: $x_{i,100} = i \max(|x_{i,1}|, \ldots, |x_{i,99}|)$. The index of outlier was unknown to the algorithm for estimating M-Posterior. We assumed that the variance of observations was known.  We used a flat (Jeffreys) prior on the mean $\mu$ and obtained its posterior distribution, which was also Gaussian with mean $\frac{\sum_{j=1}^{100} x_{ij}}{100}$ and variance $\frac{1}{100}$. We generated 1000 samples from each posterior distribution $\Pi_{100}(\cdot| \xb_i)$ for $i=1, \ldots, 25$. Setting $m=10$ in Algorithm \ref{algo:median}, we generated 1000 samples from every subset posterior $\Pi_{100,10}(\cdot |G_{j,i}),  \ j=1, \ldots, 10$ to form the empirical measures $Q_{j,i}$; here, $\cup_{j=1}^{10}G_{j,i}=\xb_i$.  Using these $Q_{j,i}$s, Algorithm \ref{algo:mpost} generated 1000 samples from the M-Posterior $\hat\Pi_{100,g}^{\sto}(\cdot| \xb_i)$ for each $i=1, \ldots, 25$. This process was replicated 50 times.  
We used Consensus MCMC \citep{Scoetal13} as a representative for scalable MCMC methods, and compared its performance with M-Posterior.    

\begin{figure}[ht]
\begin{center}
\centerline{\includegraphics[scale = 0.2]{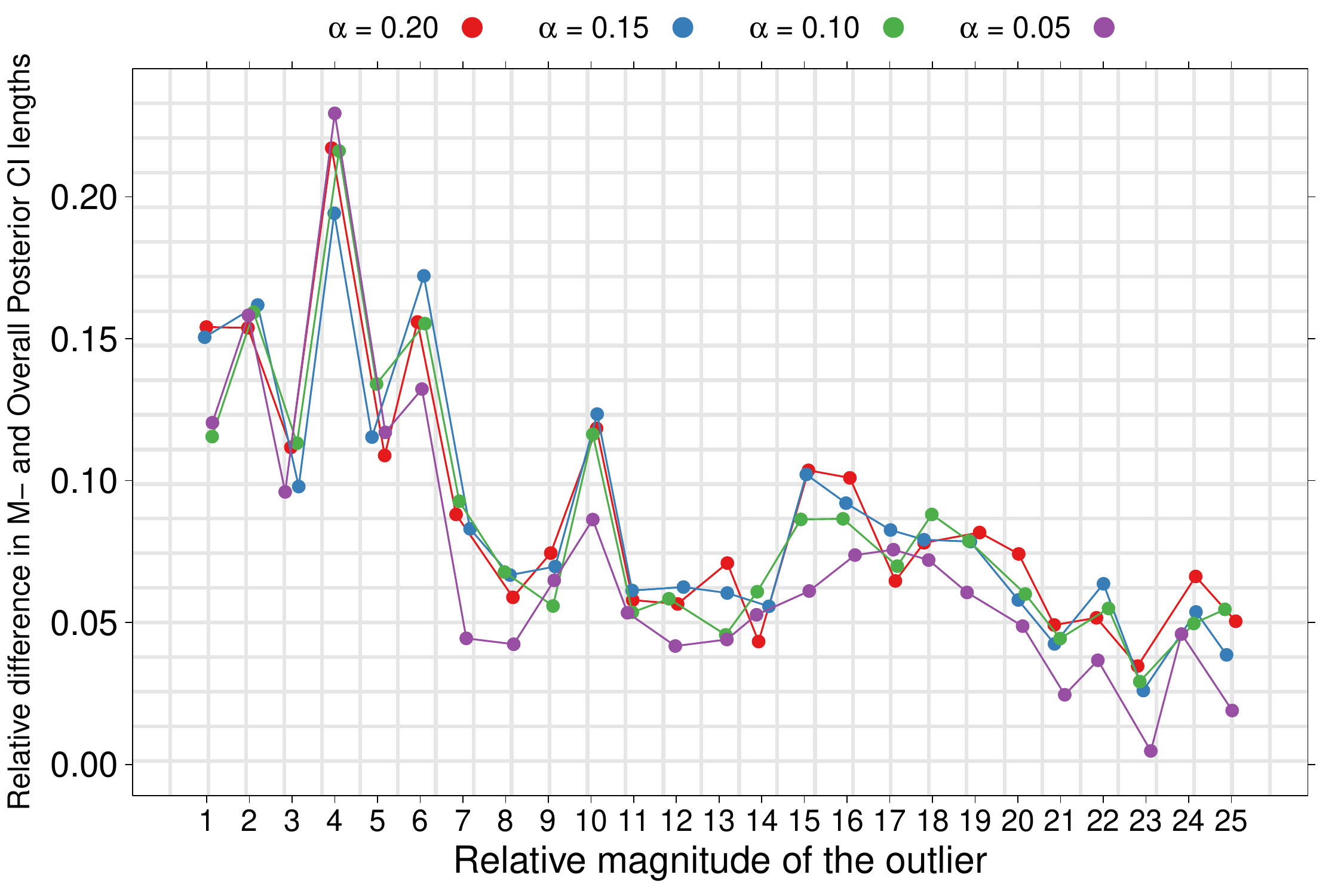}}
\caption{Calibration of uncertainty quantification of M-Posterior. The x-axis represents the outlier magnitude that increases from 1 to 25. 
The y-axis represents the relative difference between M-Posterior and overall posterior CI lengths. A value close to 0 represents that the M-Posterior CIs are well-calibrated.}
\label{fig:comp}
\end{center}
\end{figure} 

\begin{figure}[ht]
\begin{center}
\centerline{\includegraphics[scale=0.2]{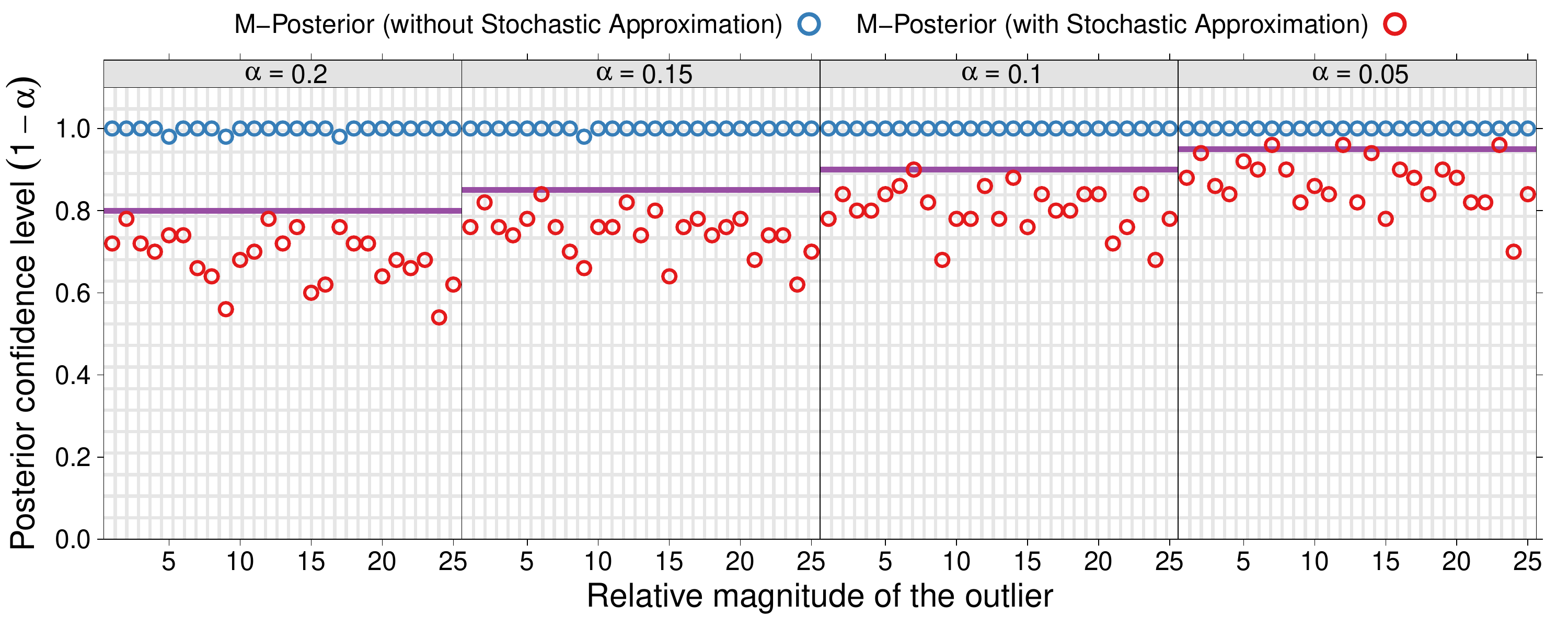}}
\caption{Effect of stochastic approximation on empirical coverage of (1-$\alpha$)100\% CIs. The x-axis represents the outlier magnitude that increases from 1 to 25. The y-axis represents the fraction of times the CIs of M-Posteriors with and without stochastic approximation include the true mean over  50 replications. The horizontal lines (in violet) show the theoretical frequentist coverage.}
\label{sto-cov}
\end{center}
\end{figure}

\begin{figure}[ht]
\begin{center}
  \centerline{\includegraphics[scale=0.2]{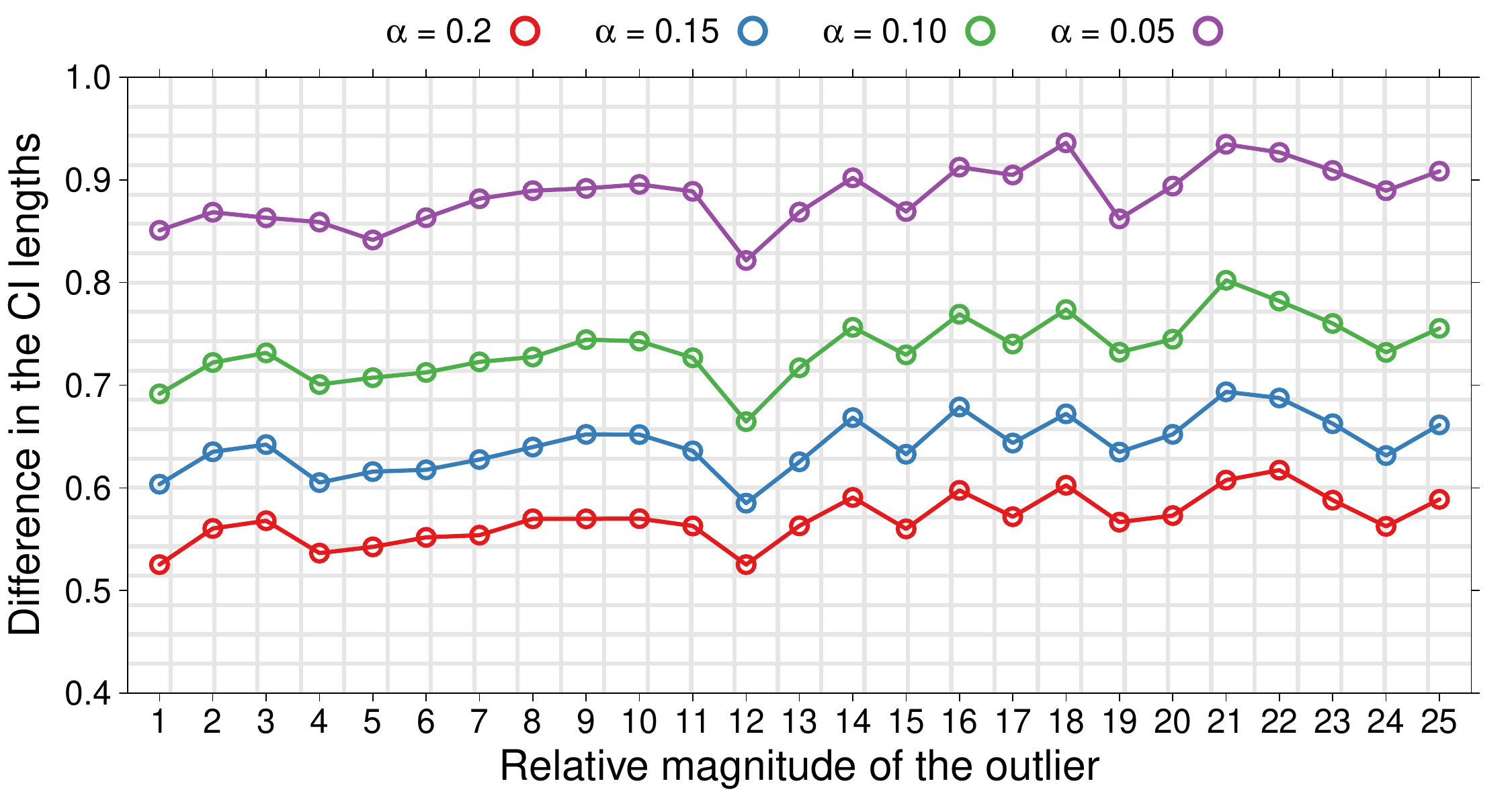}}
  \caption{Effect of stochastic approximation on the length of (1-$\alpha$)100\% CIs. The x-axis represents the outlier magnitude that increases from 1 to 25. The y-axis represents the differences in the lengths of the CIs of M-Posteriors without and with stochastic approximation.}
\label{sto-comp}
\end{center}
\end{figure} 

M-Posterior was more robust than its competitors and its performance improved with increasing magnitude of the outlier.  We compared the performance of ``consensus posterior'', the overall posterior, and the M-Posterior using the empirical coverage of (1-$\alpha$)100\% credible intervals (CIs) calculated across 50 replications for $\alpha = 0.2, 0.15, 0.10$, and 0.05. The empirical coverages of M-Posterior's CIs showed robustness to magnitude of the outlier. 
On the contrary, performance of the consensus and overall posteriors  deteriorated fairly quickly across all $\alpha$'s leading to 0\% empirical coverage as  magnitude of the outlier increased from $i=1$ to $i=25$ (Figure \ref{norm-eg1}). 
We compared uncertainty quantification of the M-Posterior with that of the overall posterior using relative lengths of their CIs, with zero value corresponding to identical lengths and a positive value to wider CIs of the M-Posterior.  
We found that widths of CIs for both posteriors were fairly similar for $i=1, \ldots, 25$, with M-Posterior's CIs being slightly wider in absence of large outliers (Figure \ref{fig:comp}). 

Stochastic approximation was important for proper calibration of uncertainty quantification. 
The empirical coverages of (1-$\alpha$)100\% CIs of the M-Posterior without stochastic approximation overcompensated for uncertainty at all levels of $\alpha$ (Figure \ref{sto-cov}). Similarly, lengths of the CIs of M-Posterior without stochastic approximation are wider than those with stochastic approximation (Figure \ref{sto-comp}). Both these observations showed that stochastic approximation led to shorter CIs for M-Posterior that had empirical coverages close to their theoretical values.

\begin{figure}[t]
  \centering
  \subfloat[Sensitivity to the choice of $m$]{
    \includegraphics[scale=0.175]{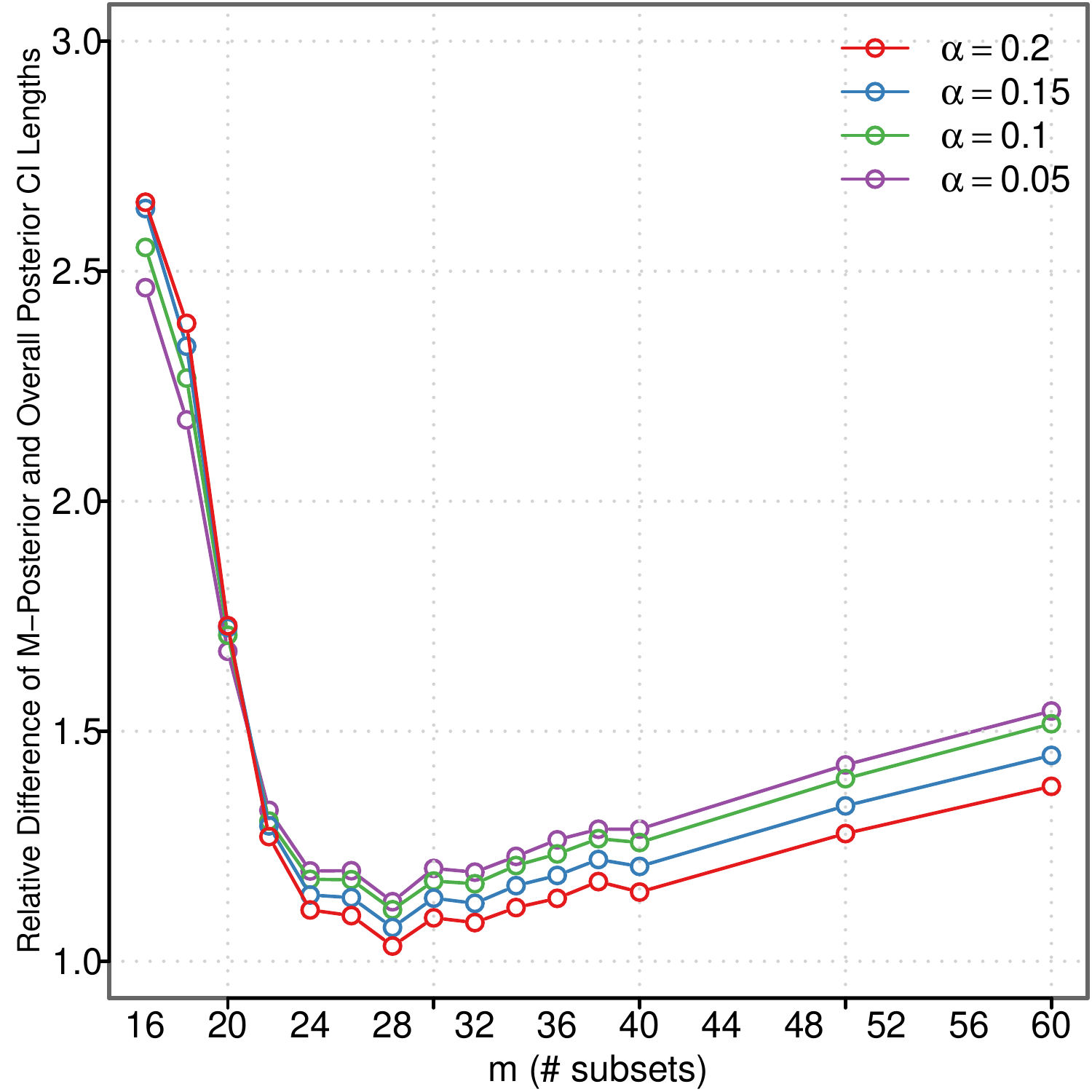}
    \label{m-cilen-sens}}\hspace{1in}
  \subfloat[Time]{
    \includegraphics[scale=0.22]{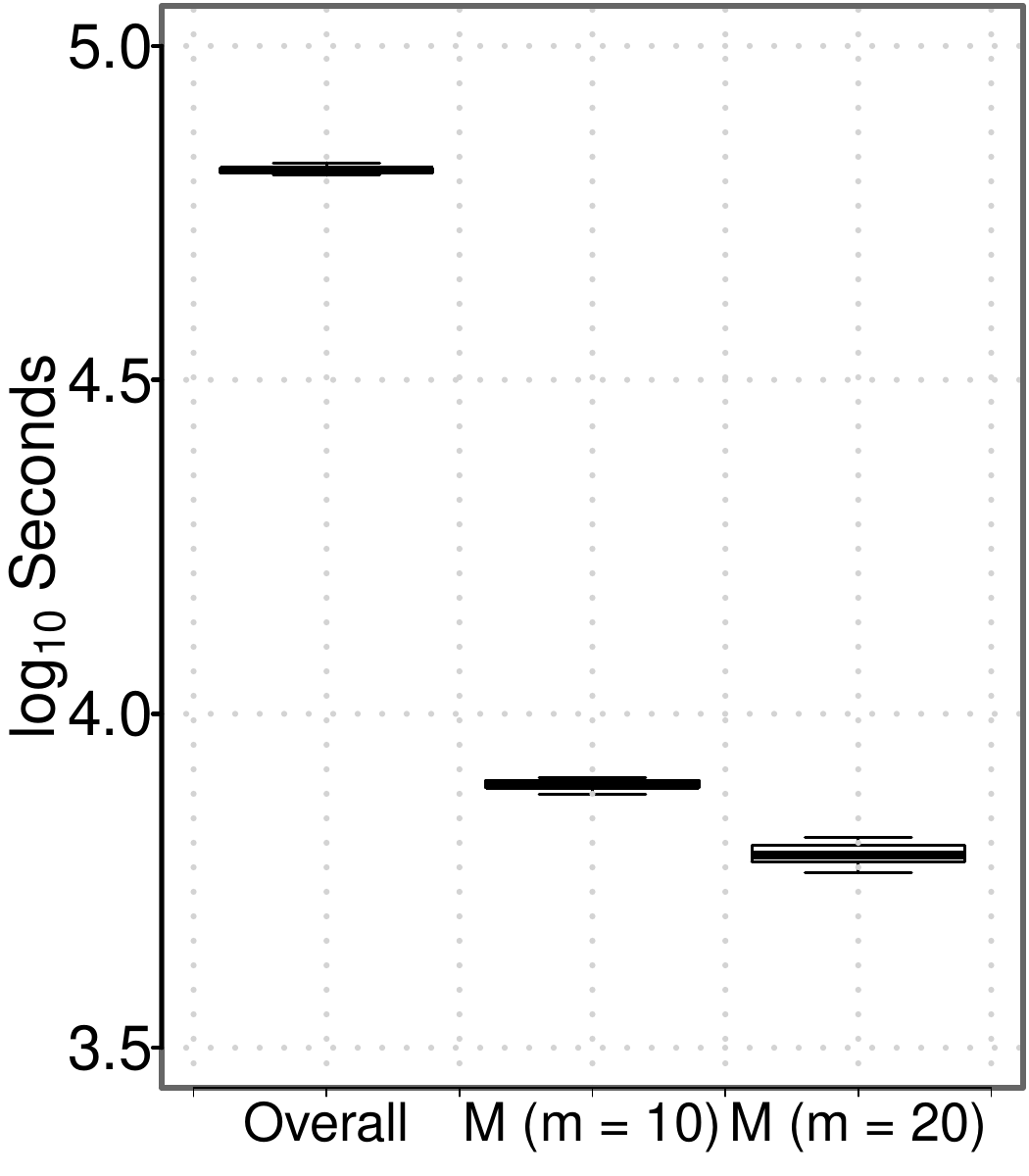}
    \label{gss-time}}
  \caption{(a) Effect of $m$ on the length of (1-$\alpha$)100\% CIs. The x-axis represents different choices of $m$. The y-axis represents the relative difference between M-Posterior and overall posterior CI lengths (median across all 25 outlier magnitudes and 50 replications).  (b) Computation time to estimate overall posterior and M-Posterior (M) with $m=10$ and $m=20$ in real data analysis.}
\end{figure}

The number of subsets ($m$) had an effect on credible interval length of the M-Posterior. We modified the simulation above and generated 1000 observations $\xb$, with the last 10 observations in $\xb$ being outliers with value $x_{j} = 25 \max(|x_{1}|, \ldots, |x_{990}|)$, $j=991, \ldots, 1000$. The simulation setup was replicated 50 times. 
M-Posteriors were obtained for $m = 16, 18, \ldots, 40, 50, 60$. Across all values of $m$, M-Posterior's CI was compared to the CI of the overall posterior after removing the outliers; 
the relative difference of M-Posterior and the overall posterior CI lengths decreases for $m\geq 22>2k$, where $k=10$ is the number of outliers, remains stable as $m$ increases to $m=38$, and grows for larger values of $m$ (Figure \ref{m-cilen-sens}). 
This demonstrates that inference based on M-Posterior was not too sensitive to the choice of $m$ for a wide range of values.

\subsection{Real data analysis: General social survey}
\label{sec:real-data-analysis}

The General Social Survey (GSS; \url{gss.norc.org}) has collected responses to questions about evolution of American society since 1972. We selected data for 9 questions from different social topics: ``happy'' (happy), ``Bible is a word of God'' (bible), ``support capital punishment'' (cap), ``support legalization of marijuana'' (grass), ``support premarital sex'' (pre), ``approve bible prayer in public schools'' (prayer), ``expect US to be in world war in 10 years'' (uswar), ``approve homosexual sex relations'' (homo), ``support abortion'' (abort). 
These questions were in the survey since 1988 and their answers were converted to two levels: \emph{yes} or \emph{no}. Missing data were imputed based on the average, resulting in a data set with approximately 28,000 respondents. 

We use a Dirichlet process (DP) mixture of product multinomial distributions, probabilistic parafac (p-parafac), to model the multivariate dependence among the 9 questions. Let $c_k \in \{\mathrm{yes}, \mathrm{no} \}$ represent the response to $k$th question, $k = 1, \ldots, 9$, then  $\pib_{c_1, \ldots, c_9}$ is the joint probability of response $c = (c_1, \ldots, c_9) $ for the 9 questions. Using $\pib_{c_1, \ldots, c_9}$, we estimated the joint probability of response to two questions $\pib_{c_i, c_j}$ for every $i$ and $j$ in $\{1, \ldots, 9\}$; see section \ref{app:parafac} of the appendix for the description of p-parafac generative model and sampling algorithms. 
The GSS data were randomly split into 10 test and training data sets. Samples from the overall posteriors of $\pib_{c_i, c_j}$s were obtained using the Gibbs sampling algorithm of \citet{DunXin09}. We chose $m $ as 10 and 20 and estimated  M-Posteriors for $\pib_{c_i, c_j}$s in four steps: training data were randomly split into $m$ subsets, samples from subset posteriors were obtained after modifying the original sampler using stochastic approximation, weights of subsets posteriors were estimated using Algorithm \ref{algo:mpost}, and atoms with estimated weights below $\tfrac{1}{2m}$ were removed.

M-Posterior had similar uncertainty quantification as the overall posterior while being more efficient. M-Posterior was at least 10 ($m = 20$) and 8 times ($m=10$) faster than the overall posterior and it used less than 25\% of the memory resources required by the overall posterior (Figure \ref{gss-time}). The overall posterior was more concentrated than the M-Posterior for $m=10$ and 20; however, its coverage of maximum likelihood estimators for $\pib_{c_i, c_j}$ obtained from the test data was worse than that of the two M-Posteriors (Table \ref{tab:mats}).

\begin{table}[ht]
  \caption{Empirical coverage and lengths of (1-$\alpha$)100\% credible intervals. The results are averaged across all joint probabilities $\pib_{c_i, c_j}$s and 10 folds of cross-validation. Monte Carlo errors are in the parentheses.} 
  \begin{tabular}{| r | c | c | c | c |}
    \hline
    \multicolumn{5}{|c|}{Empirical Coverage of (1-$\alpha$)100\% Credible Intervals}\\
    \hline
    $\alpha$ & $0.05$ & $0.10$ & $0.15$ &  $0.20$ \\ 
    \hline
    Overall Posterior     & 0.56 (0.10) & 0.52 (0.09) & 0.49 (0.09) & 0.46 (0.09) \\
    M-posterior ($m=10$)  & 0.89 (0.06) & 0.85 (0.07) & 0.82 (0.08) & 0.8 (0.08) \\
    M-posterior ($m=20$)  & 0.97 (0.03) & 0.94 (0.05) & 0.92 (0.05) & 0.91 (0.05) \\
    \hline
    \hline
    \multicolumn{5}{|c|}{Length of (1-$\alpha$)100\% Credible Intervals (in $10^{-2}$)}\\
    \hline
    $\alpha$ & $0.05$ & $0.10$ & $0.15$ & $0.20$ \\ 
    \hline
    Overall Posterior  & 1.2 (0.12) & 1.08 (0.11) & 1 (0.1) & 0.95 (0.1)\\
    M-posterior ($m=10$) & 2.75 (0.3) & 2.49 (0.27) & 2.31 (0.25) & 2.18 (0.23) \\
    M-posterior ($m=20$) & 3.71 (0.44) & 3.33 (0.4) & 3.09 (0.37) & 2.91 (0.35) \\
    \hline
  \end{tabular}
  \label{tab:mats}
\end{table}

\nocite{Huber2009Robust-statisti00}
\bibliographystyle{imsart-nameyear}
\bibliography{bibliography1,bibliography2,papers}

\appendix

\section{Remaining proofs.}
\subsection{Proof of Theorem 2.1}
\label{proof:med}
We will prove a slightly more general result:
\begin{theorem}
\label{thm:main1}
\text{}
\begin{description}
\item[a] Assume that $(\mb H,\|\cdot\|)$ is a Hilbert space and $\theta_0\in \mb H$.
Let $\hat\theta_1,\ldots,\hat\theta_m\in \mb H$ be a collection of independent random variables.
Let the constants $\alpha,q,\gamma$ be such that $0<q<\alpha<1/2$, and $0\leq \gamma<\frac{\alpha-q}{1-q}$.
Suppose $\eps>0$ is such that for all $j, \ 1\leq j\leq \lfloor(1-\gamma)m\rfloor+1$,
\begin{align}
\label{eq:weak_con1}
\Pr\Big(\|\hat\theta_j-\theta_0\|>\eps\Big)\leq q.
\end{align}
Let $\hat\theta_\ast=\med_g(\hat\theta_1,\ldots,\hat\theta_m)$ be the \textit{geometric median} of
$\{\hat\theta_1,\ldots,\hat\theta_m\}$.
Then
\begin{align*}
\Pr\Big(\|\hat\theta_\ast-\theta_0\|> C_\alpha\eps\Big)\leq \left[e^{(1-\gamma)\psi\left(\frac{\alpha-\gamma}{1-\gamma},q\right)}\right]^{-m},
\end{align*}
where $C_\alpha= (1-\alpha)\sqrt{\frac{1}{1-2\alpha}}$.
\item[b]
Assume that $(\mb Y,d)$ is a metric space and $\theta_0\in \mb Y$.
Let $\hat\theta_1,\ldots,\hat\theta_m\in \mb Y$ be a collection of independent random variables.
Let the constants $q,\gamma$ be such that $0<q<\frac 1 2$ and $0\leq \gamma<\frac{1/2-q}{1-q}$.
Suppose $\eps>0$ are such that for all $j,  \ 1\leq j\leq \lfloor(1-\gamma)m\rfloor+1$,
\begin{align}
\label{eq:weak_con2}
\Pr\Big(d(\hat\theta_j,\theta_0)>\eps\Big)\leq q.
\end{align}
Let $\hat\theta_\ast=\med_0(\hat\theta_1,\ldots,\hat\theta_m)$. Then
\begin{align*}
\Pr\Big(d(\hat\theta_\ast,\theta_0)> 3\eps\Big)\leq e^{-m(1-\gamma)\psi\left(\frac{1/2-\gamma}{1-\gamma},q\right)}.
\end{align*}
\end{description}

\end{theorem}
To get the bound stated in the paper, take $q=\frac{1}{7}$ and $\alpha=\frac{3}{7}$ in part \textbf{(a)} and $q=\frac{1}{4}$ in part \textbf{b}.

We start by proving part \textbf{a}.
To this end, we will need the following lemma (see lemma 2.1 in \cite{Minsker2013Geometric-Media00}):
\begin{lemma}
\label{lem:median}
Let $\mb H$ be a Hilbert space, $x_1,\ldots,x_m\in \mb H$ and let $x_\ast$ be their geometric median.
Fix $\alpha\in \l(0,\frac 12 \r)$ and assume that $z\in \mb H$ is such that $\|x_\ast-z\| > C_\alpha r$, where
\[
C_\alpha= (1-\alpha)\sqrt{\frac{1}{1-2\alpha}}
\]
and $r>0$.
Then there exists a subset $J\subseteq \l\{1,\ldots, m\r\}$ of cardinality $|J|>\alpha m$ such that for all $j\in J$, $\|x_j-z\|>r$.
\end{lemma}

Assume that event $\m E:=\l\{\|\hat\theta_\ast-\theta_0\|>C_\alpha\eps\r\}$ occurs.
Lemma \ref{lem:median} implies that there exists a subset $J\subseteq\{1,\ldots,m\}$ of cardinality $|J|\geq\alpha k$ such that
$\|\hat\theta_j-\theta_0\|>\eps$ for all $j\in J$, hence
\begin{align*}
\Pr(\m E)\leq &\Pr\l(\sum_{j=1}^m I\l\{\|\hat\theta_j-\theta_0\|>\eps\r\}>\alpha m \r)\leq \\
&
\Pr\l(\sum_{j=1}^{\lfloor (1-\gamma) m\rfloor+1} I\l\{\|\hat\theta_j-\theta_0\|>\eps\r\}>(\alpha-\gamma) m \frac{\lfloor (1-\gamma) m\rfloor+1}{\lfloor (1-\gamma) m\rfloor+1}\r)\leq \\
&
\Pr\l(\sum_{j=1}^{\lfloor (1-\gamma) m\rfloor+1} I\l\{\|\hat\theta_j-\theta_0\|>\eps\r\}>\frac{\alpha-\gamma}{1-\gamma}\big(\lfloor (1-\gamma) m\rfloor+1\big) \r).
\end{align*}
If $W$ has Binomial distribution $W\sim B(\lfloor(1 - \gamma) m\rfloor+1,q)$, then
\begin{align*}
\Pr\bigg(\sum_{j=1}^{\lfloor (1-\gamma) m\rfloor+1} I\l\{\|\hat\theta_j-\theta_0\|>\eps\r\}&>\frac{\alpha-\gamma}{1-\gamma}\big(\lfloor (1-\gamma) m\rfloor+1\big)\bigg)\leq \\
&
\Pr\l(W>\frac{\alpha-\gamma}{1-\gamma}\big(\lfloor (1-\gamma) m\rfloor+1\big)\r)
\end{align*}
(see Lemma 23 in \cite{lerasle2011robust} for a rigorous proof of this fact).
Chernoff bound (e.g., Proposition A.6.1 in \cite{Vaart1996Weak-convergenc00}), together with an obvious bound $\lfloor (1-\gamma) m\rfloor+1>(1-\gamma)m$, implies that
\[
\Pr\l(W>\frac{\alpha-\gamma}{1-\gamma}\big(\lfloor (1-\gamma) m\rfloor+1\big)\r)\leq
\exp\left(-m(1-\gamma)\psi\l(\frac{\alpha-\gamma}{1-\gamma},q\r)\right).
\]
To establish part \textbf{b}, we proceed as follows: let $\m E_1$ be the event
\[
\m E_1=\{\text{more than a half of events } d(\hat\theta_j,\theta_0)\leq \eps, \ j=1\ldots m \text{ occur}\}.
\]
Assume that $\m E_1$ occurs.
Then we clearly have $\eps_\ast\leq \eps$, where $\eps_\ast$ is defined in equation (2.2) of the paper: indeed, for any
$\theta_{j_1},\ \theta_{j_2}$ such that $d(\hat\theta_{j_i},\theta_0)\leq \eps, \ i=1,2$, triangle inequality gives $d(\theta_{j_1},\theta_{j_2})\leq 2\eps$.
By the definition of $\hat\theta_\ast$, inequality $d(\hat\theta_\ast,\hat\theta_j)\leq 2\eps_\ast\leq 2\eps$ holds for at least a half of $\{\hat\theta_1,\ldots,\hat\theta_m\}$, hence, it holds for some $\hat\theta_{\tilde j}$ with $d(\hat\theta_{\tilde j},\theta_0)\leq \eps$.
In turn, this implies (by triangle inequality)
$d(\hat\theta_\ast,\theta_0)\leq 3\eps$.
We conclude that
\[
\Pr\Big(d(\hat\theta_\ast,\theta_0)> 3\eps\Big)\leq \Pr(\m E_1).
\]
The rest of the proof repeats the argument of part \textbf{a} since
\[
\Pr(\m E_1^c)=\Pr\l(\sum_{j=1}^m I\l\{d(\hat\theta_j,\theta_0)>\eps\r\}\geq\frac{m}{2} \r),
\]
where $\m E_1^c$ is the complement of $\m E_1$.

\subsection{Proof of Theorem 3.1}
\label{sec:proof1}

By the definition  of Wasserstein distance $d_{W_1}$,
\begin{align}
\label{eq:a1}
d_{W_1,\rho}(\delta_0,\Pi_l(\cdot|\m X_l))&=\int_\Theta \rho(\theta,\theta_0) d\Pi_l(\theta| X_1,\ldots,X_l).
\end{align}
(recall that $\rho$ is the Hellinger distance).
Let $R$ be a large enough constant to be determined later.
Note that the Hellinger distance is uniformly bounded by 1.
Using \eqref{eq:a1}, it is easy to see that
\begin{align}
\label{eq:a3}
d_{W_{1,\rho}}(\delta_0,\Pi_l(\cdot|\m X_l))\leq R\eps_l+
\int\limits_{\rho(\theta,\theta_0)\geq R\eps_l}d\Pi_l(\cdot|\m X_l).
\end{align}
To this end, it remains to estimate the second term in the sum above.
We will follow the proof of Theorem 2.1 in \cite{Ghosal2000Convergence-rat00}.
Bayes formula implies that
\[
\Pi_l(\theta: \rho(\theta,\theta_0)\geq R\eps_l|\m X_l)=
\int\limits_{\rho(\theta,\theta_0)\geq R \eps_l}\frac{\prod_{i=1}^l \frac{p_\theta}{p_0}(X_i)d\Pi(\theta)}
{\int\limits_{\Theta}\prod_{i=1}^l \frac{p_\theta}{p_0}(X_i)d\Pi(\theta)}.
\]
Let
\[
A_l=\left\{ \theta: -P_{0}\left(\log \dfrac{p_{\theta}}{p_0} \right)\leq \eps_l^2, P_{0}\left(\log \dfrac{p_{\theta}}{p_0} \right)^2\leq \eps_l^2 \right\}.
\]
For any $C_1>0$, Lemma 8.1 \cite{Ghosal2000Convergence-rat00} yields
\[
\Pr\left\{\int\limits_{\Theta}\prod_{i=1}^l \frac{p_\theta}{p_0}(X_i)dQ(\theta)\leq \exp\big(-(1+C_1)l\eps_l^2\big)\right\}\leq \frac{1}{C_1^2l\eps_l^2} .
\]
for every probability measure $Q$ on the set $A_l$.
Moreover, by the assumption on the prior $\Pi$,
\[
\Pi(A_l)\geq \exp\left( -Cl\eps_l^2\right).
\]
Consequently, with probability at least $1-\frac{1}{C_1^2l\eps_l^2}$,
 \begin{align*}
\int\limits_{\Theta}\prod_{i=1}^l \frac{p_\theta}{p_0}(X_i)d\Pi(\theta)\geq \exp\Big(-(1+C_1)l\eps_l^2\Big)\Pi(A_l)\geq \exp(-(1+C_1+C)l\eps_l^2).
 \end{align*}
Define the event
$B_l=\left\{ \int\limits_{\Theta}\prod\limits_{i=1}^l \frac{p_\theta}{p_0}(X_i)d\Pi(\theta)\leq  \exp\Big(-(1+C_1+C)l\eps_l^2\Big) \right\}$.

Let $\Theta_l$ be the set satisfying conditions of Theorem 3.1.
Then by Theorem 7.1 in \cite{Ghosal2000Convergence-rat00}, there exist test functions
$\phi_l:=\phi_l(X_1,\ldots,X_l)$ and a universal constant $K$  such that
\begin{align}
\label{eq:test}
&
\mb E_{P_0}\phi_l\leq 2e^{-Kl\eps_l^2},
\\
\nonumber
\sup\limits_{\theta\in \Theta_l, h(P_\theta,P_0)\geq R\eps_l}&\mb E_{P_\theta}(1-\phi_l)\leq e^{-KR^2\cdot l\eps_l^2},
\end{align}
where $KR^2-1>K$.

Note that
\begin{align*}
\Pi_l(\theta: \rho(\theta,\theta_0)\geq R \eps_l|X_1,\ldots,X_l)= \Pi_l(\theta: \rho(\theta,\theta_0)\geq R \eps_l|X_1,\ldots,X_l)(\phi_l+1-\phi_l).
\end{align*}
For the first term,
\begin{align}
\label{eq-term1}
\mb E_{P_{0}}\Big[\Pi_l(\theta: \rho(\theta,\theta_0)\geq R\eps_l|X_1,\ldots,X_l)\cdot\phi_l\Big]\leq \mb E_{P_0}\phi_l\leq 2e^{-Kl\eps_l^2}.
\end{align}
Next, by the definition of $B_l$, we have
\begin{align}
\label{eq-term2}
\nonumber &\Pi_l(\theta: \rho(\theta,\theta_0)\geq R\eps_l|\m X_l)(1-\phi_l)=\\
\nonumber
&\dfrac{\int\limits_{\rho(\theta,\theta_0)\geq R \eps_l}\prod_{i=1}^l \frac{p_\theta}{p_0}(X_i)d\Pi(\theta)(1-\phi_l)}{\int\limits_{\Theta}\prod_{i=1}^l \frac{p_\theta}{p_0}(X_i)d\Pi(\theta)}\left(I\{B_l\}+I\{B_l^c\}\right)\\
&\leq I\{B_l\}+e^{(1+C_1+C)l\eps_l^2} \int\limits_{\rho(\theta,\theta_0)\geq  R\eps_l}\prod_{i=1}^l \frac{p_\theta}{p_0}(X_i)d\Pi(\theta)(1-\phi_l).
\end{align}

To estimate the second term of last equation, note that
\begin{align}
\label{eq-term3}
&\nonumber
\mb E_{P_{_0}} \int\limits_{\rho(\theta,\theta_0)\geq R\eps_l}\prod_{i=1}^l \frac{p_\theta}{p_0}(X_i)d\Pi(\theta)(1-\phi_l)\leq \\
&\nonumber
\mb E_{P_{_0}}\Bigg(\int\limits_{\theta\in \Theta\backslash\Theta_l}\prod_{i=1}^l \frac{p_\theta}{p_0}(X_i)d\Pi(\theta)(1-\phi_l)+
\int\limits_{\{\Theta_l\cap \rho(\theta,\theta_0)\geq R\eps_l\}}\prod_{i=1}^l \frac{p_\theta}{p_0}(X_i)d\Pi(\theta)(1-\phi_l)\Bigg)\leq\\
&
\Pi(\Theta\backslash\Theta_l)+\int \limits_{\{\Theta_l\cap \rho(\theta,\theta_0\}) \geq R\eps_l}\mb E_{P_0} \left(\prod_{i=1}^l \frac{p_\theta}{p_0}(X_i)d\Pi(\theta)(1-\phi_l)\right)\leq \\
& \nonumber
e^{-l\eps_l^2(C+4)}+e^{-KR^2 \cdot l\eps_l^2 }\leq
2e^{-l\eps_l^2(C+4)}
\end{align}
for $R\geq \sqrt{(C+4)/K}$.
Set $C_1=1$ and note that $I\{B_l\}=1$ with probability $P(B_l)\leq 1/l\eps_l^2$.
It follows from \eqref{eq-term1}, \eqref{eq-term2} and \eqref{eq-term3}  and Chebyshev's inequality that for any  $t>0$
\begin{align*}
\Pr\Big(\Pi_l(\theta: \rho(\theta,\theta_0)\geq R\eps_l|\m X_l)\geq t\Big)&
\leq \Pr(B_l)+\dfrac{2e^{-Kl\eps_l^2}}{t}+ \dfrac{2\exp\left( -2l \eps_l^2 \right)}{t}\\
&\leq \frac{1}{l\eps_l^2}+\dfrac{2e^{-Kl\eps_l^2}}{t}+ \dfrac{2\exp\left( -2l \eps_l^2 \right)}{t}.
\end{align*}

Finally, for a constant $\tilde K=\min(K/2,1)$ and $t=e^{-\tilde K l\eps_l^2}$, we obtain
\begin{align*}
\Pr\Big(\Pi_l(\theta: \rho(\theta,\theta_0)\geq R\eps_l|\m X_l)\geq t\Big)&\leq \frac{1}{l\eps_l^2}+2e^{-Kl\eps_l^2/2}+ 2\exp\left( -l \eps_l^2 \right)\\
&\leq \frac{1}{l\eps_l^2}+4e^{-\tilde K l\eps_l^2},
\end{align*}
which yields the result.

\subsection{Proof of Theorem 3.8}
\label{proof:3.8}

From equation (3.7) in the paper and proceeding as in proof of Theorem 3.2, we get that
\[
\l\|\delta_0-\Pi_l(\cdot|\m X_l)\r\|_{\m F_k}\leq R\eps_l^{1/\gamma}+
\int\limits_{\rho_k(\theta,\theta_0)\geq R\eps_l^{1/\gamma}}\rho_k(\theta,\theta_0)d\Pi_l(\cdot|\m X_l).
\]
		
By H\"{o}lder's inequality, 
\begin{align*}
\int\limits_{\rho_k(\theta,\theta_0)\geq R\eps_l^{1/\gamma}}&\rho_k(\theta,\theta_0)d\Pi_l(\cdot|\m X_l)\leq \\
&
\l[\int\limits_{\Theta}\rho^w_k(\theta,\theta_0)d\Pi_l(\cdot|\m X_l)\r]^{1/w}
\l[\int\limits_{\rho_k(\theta,\theta_0)\geq R\eps_l^{1/\gamma}}d\Pi_l(\cdot|\m X_l)\r]^{1/q}
\end{align*}
with $w>1$ and $q=\frac{w}{w-1}$.
		
Define the event
\[
B_l=\left\{ \int\limits_{\Theta}\prod\limits_{i=1}^l \frac{p_\theta}{p_0}(X_i)d\Pi(\theta)\geq  \exp\Big(-(1+C_1+C)l\eps_l^2\Big) \right\}.
\]
Following in the proof of Theorem 3.1, we note that 
$\Pr(B_l)\geq1-\frac{1}{C_1^2l\eps_l^2}$ (where constants $C,C_1$ are the same as in the proof of Theorem 3.1). 
Also, note that 
\begin{align*}
\mb E_{P_0}\l[\int_\Theta \rho^w_k(\theta,\theta_0)\prod_{i=1}^l \frac{p_\theta}{p_0}(X_i)d\Pi(\theta)\r] = 
\int_\Theta \rho^w_k(\theta,\theta_0) d\Pi(\theta)
\end{align*}

Hence, with probability $\geq 1-e^{-\tilde K l\eps_l^2}$, 
\[
\int_\Theta \rho^w_k(\theta,\theta_0)\prod_{i=1}^l \frac{p_\theta}{p_0}(X_i)d\Pi(\theta)\leq 
e^{\tilde Kl\eps_l^2}\int_\Theta \rho^w_k(\theta,\theta_0) d\Pi(\theta), 
\]
where $\tilde K=\min(K/2,1)$ is the same universal constant as in the proof of Theorem 3.1. 
Writing 
\begin{align*}
\l[\int\limits_{\Theta}\rho^w_k(\theta,\theta_0)d\Pi_l(\cdot|\m X_l)\r]^{1/w}=
\l[\frac{\int_\Theta \rho^w_k(\theta,\theta_0)\prod_{i=1}^l \frac{p_\theta}{p_0}(X_i)d\Pi(\theta)}
	{\int_\Theta \prod_{i=1}^l \frac{p_\theta}{p_0}(X_i)d\Pi(\theta)}\r]^{1/w},
\end{align*}
we deduce that with probability $\geq 1-e^{-\tilde K l\eps_l^2}-\frac{1}{l\eps_l^2}$ (where we set $C_1:=1$),
\begin{align*}
\l[\int\limits_{\Theta}\rho^w_k(\theta,\theta_0)d\Pi_l(\cdot|\m X_l)\r]^{1/w}\leq 
e^{\frac{2+C+\tilde K}{w}l\eps_l^2}\l[\int_\Theta \rho^w_k(\theta,\theta_0) d\Pi(\theta)\r]^{1/w}.
\end{align*}
 
By Theorem 7.1 in \cite{Ghosal2000Convergence-rat00}, there exist test functions
$\phi_l:=\phi_l(X_1,\ldots,X_l)$ and a universal constant $K$  such that
\begin{align*}
&
\mb E_{P_0}\phi_l\leq 2e^{-Kl\eps_l^2},
\\
\nonumber
\sup\limits_{\theta\in \Theta_l, d(\theta,\theta_0)\geq \tilde R\eps_l}&\mb E_{P_\theta}(1-\phi_l)\leq e^{-K\tilde R^2\cdot l\eps_l^2},
\end{align*}
where $K\tilde R^2-1>K$ and $d(\cdot,\cdot)$ is the Hellinger or Euclidean distance. 
It immediately follows from Assumption 3.4 that 
\[
\l\{ \theta: \rho(\theta,\theta_0)\geq \tilde C_0(\theta_0) R^\gamma\eps_l  \r\}\supseteq \l\{ \theta: \rho_k(\theta,\theta_0)\geq R \eps_l^{1/\gamma}  \r\},
\]
hence the test functions $\phi_l$ (for $\tilde R:=\tilde C_0(\theta_0)R^\gamma)$ satisfy  
\begin{align}
\label{eq:test2}
&
\mb E_{P_0}\phi_l\leq 2e^{-Kl\eps_l^2},
\\
\nonumber
\sup\limits_{\theta\in \Theta_l, \rho_k(\theta,\theta_0)\geq R\eps_l^{1/\gamma}}&\mb E_{P_\theta}(1-\phi_l)\leq e^{-KR^2\cdot l\eps_l^2}.
\end{align}
Repeating the steps of the proof of Theorem 3.1, we see that if $R$ is chosen large enough, then
\[
\int\limits_{\rho_k(\theta,\theta_0)\geq R\eps_l^{1/\gamma}}d\Pi_l(\cdot|\m X_l)\leq e^{-\tilde K l\eps_l^2}
\]
with probability $\geq 1- \frac{1}{l\eps_l^2}-4e^{-\tilde K l\eps_l^2}.$ 
Combining the bounds, we obtain that, for $w:=\frac{4}{3}+\frac{4+2C}{3\tilde K}$, with probability 
$\geq 1-\frac{1}{l\eps_l^2}-5e^{-\tilde K l\eps_l^2}$,
\[
\int\limits_{\rho(\theta,\theta_0)\geq R\eps_l^{1/\gamma}}\rho_k(\theta,\theta_0)d\Pi_l(\cdot|\m X_l)\leq 
e^{-\tilde K l\eps_l^2/2}\l[\int_\Theta \rho^w_k(\theta,\theta_0) d\Pi(\theta)\r]^{1/w},
\]
hence the result follows.

\subsection{Proof of Theorem 3.11}
\label{proof:dist2}

The proof strategy is similar to Theorem 3.1.
Note that
\begin{align}
\label{eq:zy1}
&
d_{W_{1,\rho}}(\delta_0,\Pi_{n,m}(\cdot|\m X_l))\leq R\eps_l+
\int\limits_{h(P_\theta,P_0)\geq R\eps_l}d\Pi_l(\cdot|\m X_l),
\end{align}
where $h(\cdot,\cdot)$ is the Hellinger distance. 

Let $\m E_l:=\{\theta: h(P_\theta,P_0)\geq R\eps_l\}$.
By the definition of $\Pi_{n,m}$, we have
\begin{align}
\label{eq:b1}
\Pi_{n,m}(\m E_l|\m X_l)= \frac{\int\limits_{\m E_l}\left(\prod_{j=1}^l \frac{p_\theta}{p_0}(X_j)\right)^md\Pi(\theta)}
{\int\limits_{\Theta}\left(\prod_{j=1}^l \frac{p_\theta}{p_0}(X_j)\right)^m d\Pi(\theta)}.
\end{align}
To bound the denominator from below, we proceed as before. Let
\[
\Theta_l=\left\{ \theta: -P_0\left(\log \dfrac{p_{\theta}}{p_0} \right)\leq \eps_l^2, P_0\left(\log \dfrac{p_{\theta}}{p_0} \right)^2\leq \eps_l^2 \right\}.
\]
Let $B_l$ be the event defined by
\[
B_l:=\left\{\int\limits_{\Theta_l}\left(\prod_{i=1}^l \frac{p_\theta}{p_0}(X_i)\right)^m dQ(\theta)\leq \exp(-2m  l\eps_l^2)\right\},
\]
where $Q$ is a probability measure supported on $\Theta_l$.
Lemma 8.1 in \cite{Ghosal2000Convergence-rat00} yields that $\Pr(B_l)\leq \frac{1}{ l\eps_l^2}$ for any $Q$, in particular, for the conditional distribution
$\Pi(\cdot|\Theta_l)$.
We conclude that
\[
\int\limits_{\Theta}\left(\prod_{j=1}^l \frac{p_\theta}{p_0}(X_j)\right)^m d\Pi(\theta)\geq \Pi(\Theta_l)\exp(-2ml\eps_l^2)\geq
\exp(-(2m+C)l\eps_l^2).
\]

To estimate the numerator in (\ref{eq:b1}), note that if Theorem 3.10 holds for $\gamma=\eps_l$, then it also holds for $\gamma=L\eps_l$ for any $L\geq 1$.
This observation implies that
\[
\sup_{\theta\in \m E_l}\left(\prod_{j=1}^l \frac{p_\theta}{p_0}(X_j)\right)^{m}\leq e^{-c_1R^2 ml \eps_l^2}
\]
with probability $\geq 1-4e^{-c_2R^2 l\eps_l^2}$, hence
\begin{align*}
\int\limits_{\m E_l}\left(\prod_{j=1}^l \frac{p_\theta}{p_0}(X_j)\right)^{m} d\Pi(\theta)\leq e^{-c_1R^2 m l\eps_l^2}
\end{align*}
with the same probability.
Choose $R=R(C)$ large enough so that $c_1 m R^2\geq 3m+C$.
Putting the bounds for the numerator and denominator of (\ref{eq:b1}) together, we get that with probability
$\geq 1-\frac{1}{ l\eps_l^2}-4e^{-c_2R^2 l \eps_l^2}$,
\[
\Pi_{n,m}(\m E_l| \m X_l)\leq e^{-ml\eps_l^2}.
\]
The result now follows from (\ref{eq:zy1}).

\subsection{Proof of Proposition 3.13}
\label{proof:3.13}

The proof follows a standard pattern: on the first step, we show that it is enough to consider the posterior obtained from the prior restricted to a large compact set, and then proving the theorem for the prior with compact support. 
The second part mimics the classical argument exactly (e.g., see \cite{van2000asymptotic}). 

To show that one can restrict the prior to the compact set, it is enough to establish that for $R$ large enough and 
$\m E_l:=\{\theta: \| \theta - \theta_0 \|\geq \frac{R}{\sqrt l}\}$,
\begin{align}
\label{eq:b11}
\Pi_{n,m}(\m E_l|\m X_l)= \frac{\int\limits_{\m E_l}\left(\prod_{j=1}^l \frac{p_\theta}{p_0}(X_j)\right)^md\Pi(\theta)}
{\int\limits_{\Theta}\left(\prod_{j=1}^l \frac{p_\theta}{p_0}(X_j)\right)^m d\Pi(\theta)}
\end{align}
can be made arbitrarily small. 
This follows from the inclusion 
\[
\m E_l\subseteq \l\{ \theta: h(P_\theta,P_{\theta_0})\geq \tilde C\frac{R}{\sqrt l} \r\}
\] 
(due to the assumed inequality between Hellinger and Euclidean distances) and the bounds for the numerator and the denominator of (\ref{eq:b11}) established in the proof of Theorem 3.11.

\section{Numerical simulation: additional examples and details.} 


\subsection{Probabilistic parafac model}
\label{app:parafac}

The generative model of p-parafac has two levels. First, prior probabilities of latent classes and parameters are sampled. Discrete random measure $\nu(\cdot) = \sum_{h=1}^{\infty} \nu_{h} \delta_{h}(\cdot)$ is generated using the stick-breaking construction of DP
\begin{align*}
  V_h \mid \alpha \sim \mathrm{Beta}(1, \alpha) \text{ and } \nu_h = V_h \prod_{l<h} (1 - V_l) \text{ for } h = 1, \ldots, \infty,
\end{align*}  
where $\nu_h$ is the prior probability of responders responses belonging to the latent class $h$. The prior probability of a response to $k$th question depending on the latent class $h$ 
\begin{align*}
  \Psi_{h}^{k}   = (\psi_{h;\mathrm{yes}}^{k}, \psi_{h;\mathrm{no}}^{k}) \mid \alpha_1, \alpha_2 \sim \mathrm{Dirichlet}(\alpha_1, \alpha_2), \quad k = 1, \ldots, 9 \text{ and } h = 1, \ldots, \infty.
\end{align*}
Second, latent variables and parameters are generated, which are specific to responders in the training data. The latent class of $n$th responder
\begin{align*}
  z_n \mid \nu(\cdot) \sim \sum_{h=1}^{\infty} \nu_{h} \delta_{h}(\cdot).
\end{align*}
Finally, the response of $n$th responder for question $k$ 
\begin{align*}
  y_{nk} \mid \Psib_{1}^{k}, \ldots, \Psib_{\infty}^{k} , z_n  \sim \mathrm{Multinomial} \big(\{\mathrm{yes}, \mathrm{no} \}, \Psib_{z_n}^{k}\big). 
\end{align*}
This generative model in turn implies that 
\begin{align*}
  &P(y_{n1} = c_1, \ldots, y_{nk}, \ldots, y_{n9} = c_9)=\pib_{c_1, \ldots, c_9} = \sum_{h=1}^{\infty} \nu_h \prod_{k=1}^9 \psi_{h;c_k}^{k}, 
\end{align*}
 where $c_k \in \{\mathrm{yes}, \mathrm{no}\}$, $k \in \{1, \ldots, 9\}$, and $\nu_h$s and $\Psi^k_h$s respectively have the stick-breaking and Dirichlet prior distributions.  The hyperparameters of this model are $\alpha$, $\alpha_1$, and $\alpha_2$. We specify Gamma prior on $\alpha$ with scale and shape parameter fixed at 1 and assume that $\alpha_1=1$ and $\alpha_2=1$. Due to the finite number of available responses, the number of latent classes is upper-bounded by a finite number \citep{IshLan01}. This formulation leads to a simple Gibbs sampler for obtaining posterior samples of $\pib_{c_1, \ldots, c_9} $ (see (5) in \cite{DunXin09} for analytic forms of the conditional distributions).  Samples of $\pib_{c_i, c_j} $ are obtained after marginalizing responses for remaining 7 questions from $\pib_{c_1, \ldots, c_9} $. Appendix D in \citet{Srietal15b} shows a general scheme for sampling from subset posteriors that are obtained after modifying the original p-parafac sampler using stochastic approximation.

 All sampling algorithms were implemented in Matlab. The samplers ran for
 10,000 iterations and every fifth sample was collected after a burn-in of 5000. All experiments were performed on Oracle Grid Engine cluster with 2.6GHz 16 core compute nodes. Memory resources were capped at 16GB and 64GB for sampling from subset and overall posteriors, respectively.

\begin{figure}[t]
  \centering
  \subfloat[Support capital punishment (cap) and legalization of marijuana (grass)]{
    \includegraphics[width=\textwidth]{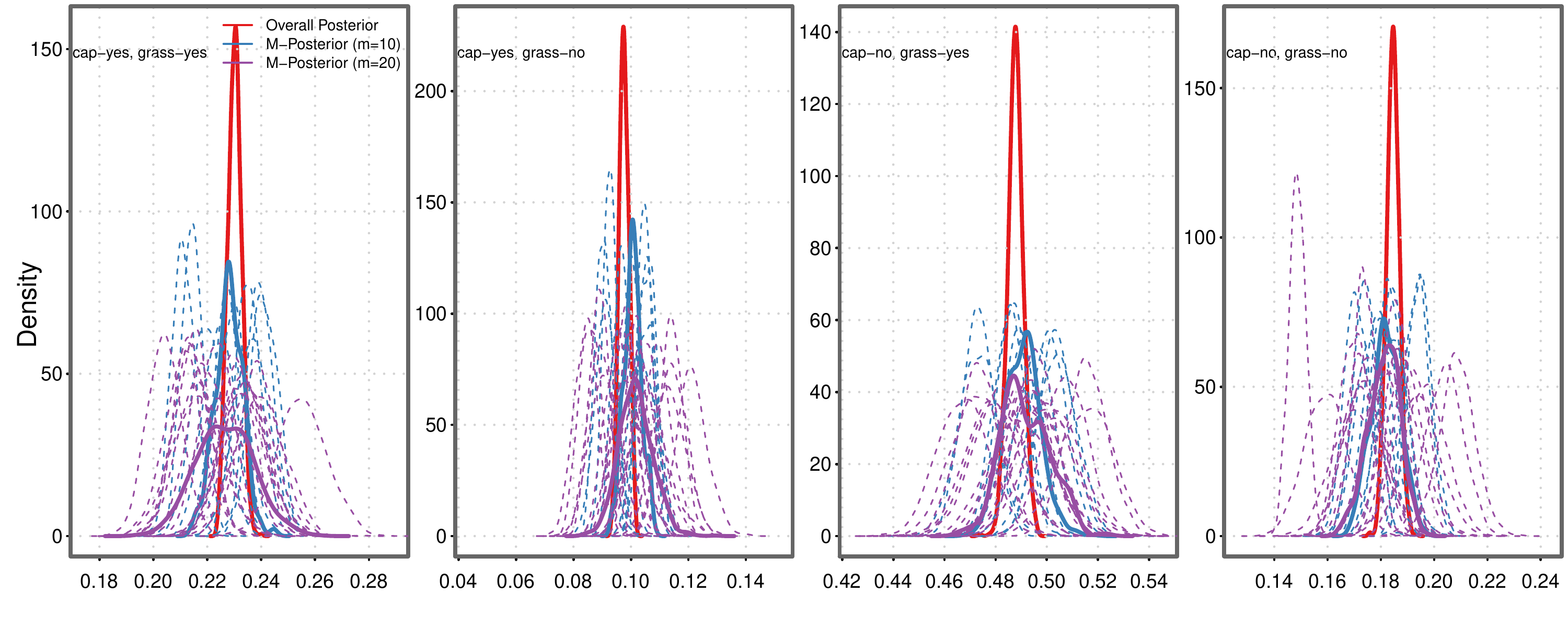}
    \label{m-cilen-sens1}}\\
  \subfloat[Expect US to be in world war in 10 years (uswar) and support abortion (abort)]{
    \includegraphics[width=\textwidth]{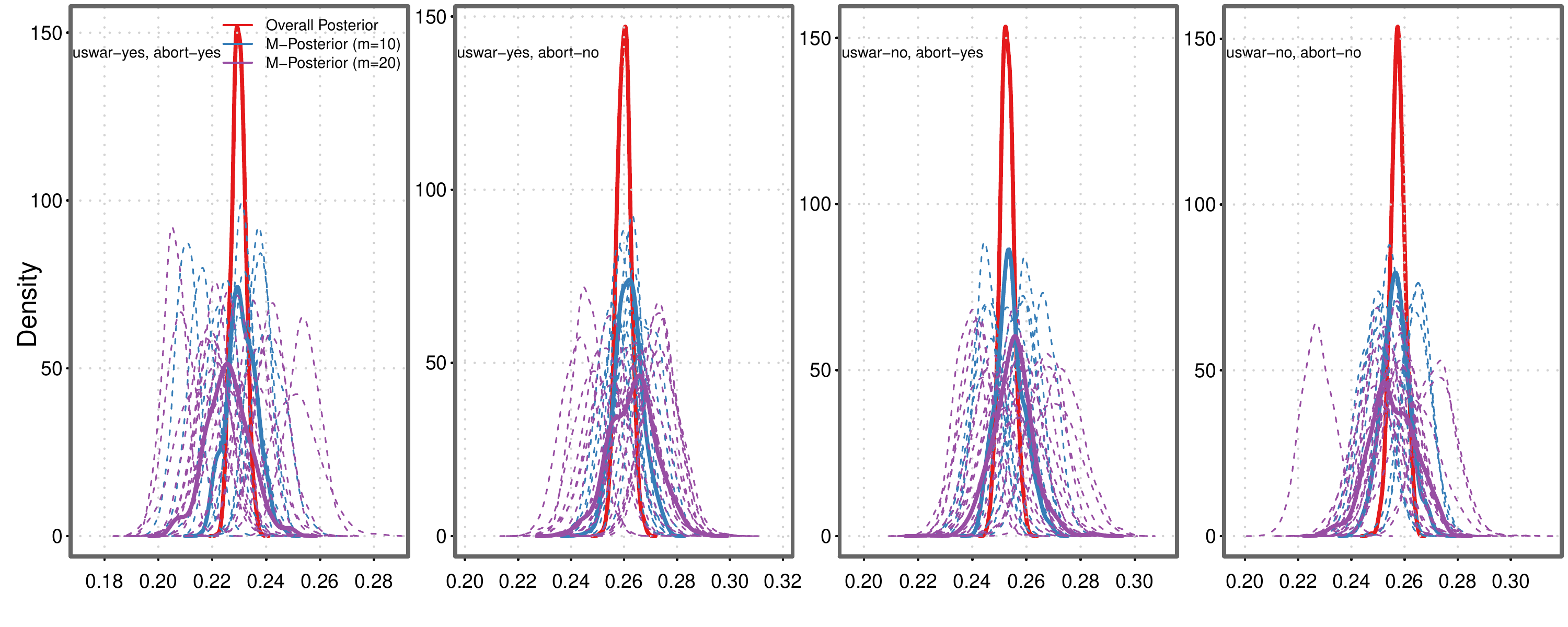}
    \label{gss-time1}}
  \caption{Kernel density estimators of subset posteriors (dashed), M-Posterior with $m=10$ (blue) and $m=20$ (violet), and the overall posterior (red) for the four responses. In particular, the model ``recognizes'' that ``expect US to be in world war in 10 years'' (uswar) and ``support abortion'' (abort) should be independent.}
\end{figure}

\subsection{Dirichlet process mixture model sampling for the stochastic approximation}
\label{section:dirichlet}

As we have mentioned, using the stochastic approximation in place of the subset posterior does not lead to increased sample complexity in many cases. 
Many Bayesian models involve hierarchical exponential family specifications, in which case conditional distributions in Gibbs sampling or acceptance probabilities in Metropolis-Hastings algorithms can be  trivially modified to account for the weighted likelihood.  
Here, we present one particular example of the Dirichlet process mixture model. 
We augment the original sampling model with latent variables and raise the complete data likelihood to an appropriate power. 
We have observed excellent performance with this approach in other contexts, including the
matrix completion example.

Let the data $\Xcal_n = \{X_1, \ldots, X_n\}$ be i.i.d. samples from a probability measure $P_{0}$ on $\RR^D$ that  has density $p_{0}$ with respect to the Lebesgue measure. A popular model for inference on $P_0$ using Dirichlet process mixtures assumes that 
\begin{align}
&\nonumber
  X_i \mid P \sim P, \quad P(dx) = \int k_{\theta} (x) G(d \theta), \\ 
 &
 G \sim DP(\alpha P_{\theta_0}), \text{ independently for }i=1, \ldots, n, 
 \label{eg0}
\end{align}
where $k_{\theta} (x) $ is a kernel such that $\int k_{\theta} (x) dx =1$ and $\theta \subseteq \RR^p$. We will focus on a special case of this setup with $D=1$, $p=1$, $k_{\theta} (x) = \tfrac{1}{\sqrt{2 \pi \sigma^2}} \exp\left(-\tfrac{(x - \theta)^2} {2 \sigma^2} \right)$. For inference on $P$ using MCMC, each $X_i$ is augmented with latent variable $Z_i$ and the hierarchical model in \eqref{eg0} is written using stick breaking representation \cite{sethuraman1994constructive} as
\begin{align}
  P(d x_i) \mid \sigma, Z_i &= \frac{1} {\sqrt{2 \pi \sigma^2}} \exp\left(-\frac{(x_i - \theta_{z_i})^2} {2 \sigma^2} \right),\quad Z_i \sim \sum_{h=1}^{\infty} \nu_h \delta_h, \quad \nu_h = V_h \prod_{q < h} (1 - V_q), \nonumber \\  
  V_h &\sim \text{Beta}(1, \alpha), \quad \theta_h \sim \text{Normal} (0, \sigma_{\theta}^2), \quad \text{ independently for } h = 1, \ldots, \infty   \label{eg1}
\end{align}
and $i=1, \ldots, n$; $\sigma^2$ and $\alpha$ are respectively assigned Inverse-Gamma($a_{0}, b_{0}$) and Gamma($a_{\alpha}$, $b_{\alpha}$) priors. Assume that data are partitioned into $m$ subsets of size $l$ such that data on subset $j$ are $\Xcal_j = \{X_{j1}, \ldots, X_{jl}\}$. The complete data likelihood for subset $j$ after stochastic approximation is
\begin{align*}
  l_j^m(\theta_1, \ldots, \theta_{K^*}) = \left( \prod_{i=1}^l \prod_{h=1}^{K^*}  \left(\frac{1} {\sqrt{2 \pi \sigma^2}} \exp \left( -\frac{(x_{ji} - \theta_h)^2 } {2 \sigma^2} \right)\right)^{1(Z_{ji} = h)}\right)^{m},
\end{align*}
where $1(\cdot)$ is an indicator function, $K^*$ is the maximum number of atoms in the stick breaking representation for $G$, and $l_j^m( \theta_1, \ldots, \theta_{K^*} )$ also depends on latent variables and $\sigma$. Full conditionals of latent variables and unknown parameters are tractable in terms of standard distributions. The Gibbs sampler iterates between the following five steps:
\begin{enumerate}
\item Sample $\theta_h \mid \text{rest}$ from Normal($\mu_h$, $\sigma_h^2$) for $h=1, \ldots, K^*$, where
  \begin{align*}
    \sigma^2_h = \left( m \sigma^{-2} \sum_{i=1}^l 1(z_{ji} = h) + \sigma_{\theta}^{-2} \right)^{-1}, \quad \mu_h = \frac{m \sigma^2_h }{\sigma^{2}  } \sum_{i=1}^l 1(z_{ji} = h) x_{ji}.
  \end{align*}
\item Sample $Z_{ji} \mid \text{rest}$ for $i=1, \ldots, l$ from the categorical distribution 
  \begin{align*}
    &Z_{ji}  \mid \text{rest} \sim \sum_{h=1}^{K^*} p_{jh} \delta_{h}, \quad p_{jh} = \frac{w_{jh}}{\sum_{h=1}^{K^*} w_{jh}},\quad 
    w_{jh} = \nu_h \exp\left(-\frac{m(x_{ji} - \theta_{h})^2} {2 \sigma^2} \right). 
  \end{align*}
\item Sample $\sigma^2 \mid \text{rest}$ from Inverse-Gamma($a_{\sigma}, b_{\sigma}$), where
  \begin{align*}
    a_{\sigma} = \frac{ml}{2} + a_0, \quad b_{\sigma} = \frac{m} {2} \sum_{i=1}^{l} \sum_{h=1}^{K^*} 1 (z_{ji} = h) (x_{ji} - \theta_h)^2  + b_0.
  \end{align*}
\item Sample $V_h \mid \text{rest}$ from Beta($1 + \sum_{i=1}^l 1(z_{ji} = h)$, $\alpha + \sum_{i=1}^l 1(z_{ji} > h)$) for $h = 1, \ldots, K^*$. 
\item Sample $\alpha \mid \text{rest}$ from Gamma($a_{\alpha} + K^*$, $b_{\alpha} - \sum_{h=1}^{K^*} \log(1 - V_h)$).
\end{enumerate}
We fix $\sigma_{\theta} = 100, a_{\alpha} = b_{\alpha} =1, $ and $a_0 = b_0=0.01$ following standard conventions.

\subsection{Selection of the optimal number of subsets $m$}
\label{section:selection}

The following heuristic approach picks the median among the candidate $M$-posteriors. 
Namely, start by evaluating the M-Posterior for each $m$ in the range of candidate values $[m_1,m_2]$:
\[
\begin{array}{c}
\hat\Pi_{n,m_1}^g:=\med_g(\Pi_n^{(1)},\ldots,\Pi_n^{(m_1)}),\\
\hat\Pi_{n,m_1+1}^g:=\med_g(\Pi_n^{(1)},\ldots,\Pi_n^{(m_1+1)}),\\
\vdots\\
\hat\Pi_{n,m_2}^g:=\med_g(\Pi_n^{(1)},\ldots,\Pi_n^{(m_2)})
\end{array}
\]
and choose $m_\ast\in[m_1,m_2]$ such that 
\begin{align}
\label{eq:automatic_m}
&
\hat\Pi^g_{n,m_\ast}=\med_0\left(\hat\Pi_{n,m_1}^g,\hat\Pi_{n,m_1+1}^g,\ldots,\hat\Pi_{n,m_2}^g\right),
\end{align}
where $\med_0$ is the \emph{metric median} defined in (2.3) in the section 2 of the paper.

\end{document}